\documentclass{amsart}
\usepackage{amsfonts}
\usepackage{amsmath}
\usepackage{amssymb}
\usepackage{amsthm}
\usepackage{esint}
\usepackage{yhmath}
\usepackage{enumerate}
\usepackage[noadjust]{cite}
\usepackage{accents}
\usepackage{mathtools}
\usepackage{thmtools}
\usepackage{thm-restate}

\newcommand{\AND}{\mathrm{\; and \;}}

\newcommand{\WHEN}{\text{ when }}

\newcommand{\sps}[1]{\left( #1 \right)}

\newcommand{\f}[2]{\tfrac{#1}{#2}}
\newcommand{\ff}[2]{\frac{#1}{#2}}
\newcommand{\n}[1]{\lVert #1 \rVert}
\newcommand{\ns}[1]{\left\lVert #1 \right\rVert}
\renewcommand{\a}[1]{\lvert #1 \rvert}
\newcommand{\as}[1]{\left\lvert #1 \right\rvert}
\newcommand{\eval}[1]{\left. #1 \right\rvert}
\ifdefined\stretchybydefault

\newcommand{\floor}[1]{\left \lfloor #1 \right \rfloor}

\else

\newcommand{\floor}[1]{\lfloor #1 \rfloor}

\fi
\newcommand{\tn}[1]{{\left\vert\kern-0.25ex\left\vert\kern-0.25ex\left\vert #1 
    \right\vert\kern-0.25ex\right\vert\kern-0.25ex\right\vert}}
\let\j\jap

\newcommand{\loc}{{\mathrm{loc}}}
\DeclareMathOperator*{\esssup}{ess\,sup}

\DeclareMathOperator{\Char}{char}

\DeclareMathOperator{\supp}{supp}

\DeclareMathOperator{\RealPart}{Re}
\renewcommand{\Re}{\RealPart}

\DeclareMathOperator{\curl}{curl}
\DeclareMathOperator{\grad}{grad}
\DeclareMathOperator{\divergence}{div}
\renewcommand{\div}{\divergence}

\newcommand{\R}{\mathbb{R}}

\newcommand{\C}{\mathbb{C}}

\newcommand{\g}{\mbox{\fontfamily{pzc}\selectfont g}}

\providecommand{\comment}[1]{}

\newcommand{\cp}{\circ}

\newcommand{\wc}{\rightharpoonup}

\renewcommand{\bar}{\overline}

\let\ggg\gg
\renewcommand{\ll}{\lesssim}
\renewcommand{\gg}{\gtrsim}
\let\al\alpha

\let\z\zeta
\let\g\gamma
\renewcommand{\k}{\kappa}
\let\d\delta
\let\e\epsilon

\let\th\theta
\let\Si\Sigma
\let\si\sigma
\let\ti\tilde
\let\ld\lambda
\let\lp\Delta
\let\na\nabla
\let\Ld\Lambda

\let\cd\cdot
\let\om\omega
\let\Om\Omega

\newcommand{\E}{\mathbb{E}}

\renewcommand{\Pr}{\mathbb P}

\newcommand{\pd}{\partial}
\newcommand{\pb}{\bar\partial}
\newcommand{\thnorm}[1]{|\kern-1pt|\kern-1pt| #1 |\kern-1pt|\kern-1pt|}
\theoremstyle{remark}
\newtheorem*{remark*}{Remark}
\theoremstyle{definition}
\newtheorem*{openproblem*}{Open Problem}
\newtheorem*{dfn*}{Definition}

\theoremstyle{plain}
\newtheorem*{theorem*}{Theorem}
\newtheorem*{proposition*}{Proposition}
\newtheorem*{conjecture*}{Conjecture}
\newtheorem{theorem}{Theorem}[section]

\newtheorem*{lemma*}{Lemma}
\newtheorem{lemma}[theorem]{Lemma}
\newtheorem*{corollary*}{Corollary}

\newcommand{\qq}{\quad\quad}

\let\we\wedge
\newcommand{\sm}{\setminus}

\binoppenalty=\maxdimen
\relpenalty=\maxdimen

\begin{document}
\title[Determination of an unbounded magnetic potential]{Unique determination of a magnetic Schr\"odinger operator with unbounded magnetic potential from boundary data}
\author{Boaz Haberman}
\thanks{This material is based upon work supported the National Science Foundation Graduate Research Fellowship under Grant No. DGE 1106400 and additionally by the NSF under Grant No. DMS-1440140 while the author was in residence at the Mathematical Sciences Research Institute in Berkeley, California, during the Fall 2015 semester.}
\begin{abstract}
We consider the Gel'fand-Calder\'on problem for a Schr\"odinger operator of the form $-(\na + iA)^2 + q$, defined on a ball $B$ in $\R^3$. We assume that the magnetic potential $A$ is small in $W^{s,3}$ for some $s>0$, and that the electric potential $q$ is in $W^{-1,3}$. We show that, under these assumptions, the magnetic field $\curl A$ and the potential $q$ are both determined by the Dirichlet-Neumann relation at the boundary $\pd B$. The assumption on $q$ is critical with respect to homogeneity, and the assumption on $A$ is nearly critical. Previous uniqueness theorems of this type have assumed either that both $A$ and $q$ are bounded or that $A$ is zero.
\end{abstract}
\maketitle
\section{Boundary data for Schr\"odinger operators}
Consider a Schr\"odinger Hamiltonian of the form 
\begin{align*}
L_{A,q}u &= -(\na + iA)^2u + q u.
\end{align*}
Here $A$ represents a magnetic vector potential and $q$ represents an electric scalar potential. 

Let $\Om \subset \R^n$ be an bounded open set. Define the Dirichlet-Neumann relation by
\[\Ld_{A,q}^\Om = \{(\eval{u}_{\pd \Om}, \eval{(\pd_\nu + i \nu \cd A) u}_{\pd \Om}): u \in H^1(B) \AND L_{A,q} u = 0\},\]
where $\nu$ is the outward unit normal to $\pd \Om$.  The Gel'fand-Calder\'on problem~\cite{Gelfand1954,Calderon1980} is to determine the magnetic field $\curl A$ and the electric potential $q$ from the Dirichlet-Neumann relation $\Ld_{A,q}$. This is impossible unless the problem has a uniqueness property, namely that $\Ld_{A,q}$ uniquely determines $\curl A$ and $q$.

We are interested in proving uniqueness under minimal {\em a priori} regularity assumptions on $A$ and $q$. To avoid unnecessary technical complications, we take $\Om$ to be a ball in $\R^3$ and assume that the coefficients $A$ and $q$ are supported away from $\pd \Om$. 

\begin{restatable}{maintheorem}{thetheorem}
\label{thetheorem}
Fix $s>0$. Let $B = B(0,1)$ be the unit ball in $\R^3$. Suppose that $A_i$ and $q_i$ are supported in the smaller ball $\f 1 2 B$. If, for each $i = 1,2$, the magnetic potential $A_i$ is small in the $W^{s,3}$ norm, the electric potential $q_i$ is in $W^{-1,3}$, and $\Ld_{A_1,q_1} = \Ld_{A_2,q_2}$, then $\curl A_1 = \curl A_2$ and $q_1 = q_2$.
\end{restatable}
One physical motivation for studying this problem comes from quantum mechanics. For compactly supported potentials, the map $\Ld_{A,q}$ contains the same information as the scattering matrix at a fixed energy level. The inverse scattering problem is to determine a localized (short-range) potential from observations made at spatial infinity. The inverse scattering problem can also be posed for potentials that are exponentially decreasing rather than compactly supported. 

Unique determination of a bounded electric potential $q$ from the Dirichlet-to-Neumann map in the absence of a magnetic potential was proven by Sylvester and Uhlmann~\cite{SylvesterUhlmann1987} (see also~\cite{NachmanSylvesterUhlmann1988}). The proof is based on a density argument using complex geometrical optics (CGO) solutions in the spirit of~\cite{Calderon1980}.

The Sylvester-Uhlmann method was adapted to the case of a nonzero magnetic potential in~\cite{Sun1993}, where uniqueness was proven for $A\in C^2$ and $q \in L^\infty$, subject the requirement that $\n{\curl A}_\infty$ be small. The basic method we use in this paper is the same as that in~\cite{Sun1993}; in particular, we retain a smallness condition on the magnetic potential.

For smooth $A$, the smallness condition was removed in~\cite{NakamuraSunUhlmann1995} using a pseudodifferential conjugation technique from~\cite{NakamuraUhlmann1994}. This was improved to $A \in C^1$ in~\cite{Tolmasky1998} using symbol smoothing. In~\cite{Panchenko2002}, it was shown that, by imposing a Coulomb gauge, the result of~\cite{Sun1993} could be extended to small Dini continuous $A$ (this includes the case where $A$ is $\al$-H\"older for some  $\al>0$). The smallness condition was removed in~\cite{Salo2004} using pseudodifferential conjugation. In~\cite{KrupchykUhlmann2014}, the Coulomb gauge condition and pseudodifferential conjugation were eliminated using an argument based on Carleman estimates with slightly convex weights, and uniqueness was proven for $A \in L^\infty$. Our result requires that $A$ is small and slightly more differentiable than in~\cite{KrupchykUhlmann2014}. On the other hand, we require much less integrability for $A$ and $q$, so that our conditions on $A$ and $q$ are much closer to being scale-invariant. It does not seem that the method in~\cite{KrupchykUhlmann2014} for removing the smallness condition on $A$ extends to the case of unbounded potentials. However, we believe that a pseudodifferential conjugation argument could be used to remove the smallness condition for the result in this paper.

Another approach to the problem in the spirit of~\cite{Faddeev1976} is based on the $\pb$ method of~\cite{BealsCoifman1985}. Using this approach, the inverse scattering problem for small $A$ and $q$ in $e^{-\g\j x} C^\infty$ was solved in~\cite{KhenkinNovikov1987}. Uniqueness for $A = 0$ and $q \in e^{-\g\j{x} } L^\infty$ was proven in~\cite{Novikov1994}. Uniqueness for $(A,q) \in e^{-\g\j{x}} C^\infty$ with no smallness condition was proven in~\cite{EskinRalston1995}. A proof of uniqueness for $A = 0$ and $q \in e^{-\g\j x} L^\infty$ using a density argument in the spirit of~\cite{SylvesterUhlmann1987} was given by~\cite{Melrose1995,UhlmannVasy}. This density argument was modified to include $A \in e^{-\g\j{x}} W^{1,\infty}$ by~\cite{PaivarintaSaloUhlmann2010}.

Since the Laplacian has units $(\text{length})^{-2}$, the $L^\infty$ norms of $A$ and $q$ are not dimensionless quantities. This is undesirable from a physical point of view. For example, assuming that $A$ and $q$ are bounded excludes even subcritical potentials with $\a{x}^{-1}$ singularities at the origin (for example, a localized Coulomb-type potential).  A scale-invariant assumption in $n$ dimensions is that $A$ be in $L^n$ and that $q$ be in $L^{n/2}$ or $W^{-1,n}$ (by Sobolev embedding, $L^{n/2} \subset W^{-1,n}$). 

The Sylvester-Uhlmann argument can be adapted to the case of unbounded potentials by using $L^p$ Carleman estimates, which are analogous to Strichartz estimates~\cite{Strichartz1977} for dispersive equations. The $L^p$ Carleman estimates originate in the theory of unique continuation. They are based on Fourier restriction theorems (in particular, the Stein-Tomas theorem~\cite{Stein1993,Tomas1975} and its variants). This connection was first noticed in this context by~\cite{Hormander1983}, and was further developed in~\cite{Jerison1986,KenigRuizSogge1987,ChanilloSawyer1990,ChiarenzaRuiz1991,RuizVega1991,KochTataru2005}). 

Chanillo~\cite{Chanillo1990}, using the weighted inequalities of~\cite{ChanilloSawyer1990}, proved uniqueness in the inverse boundary value problem for $A = 0$ and compactly supported $q$  with small norm in the scale-invariant Fefferman-Phong classes $F_{>(n-1)/2}$ (including, in particular, potentials of small weak $L^{n/2}$ norm). Chanillo's paper also includes an argument of Jerison and Kenig proving uniqueness for $q \in L^{n/2+}$ with no smallness condition. This was extended to include the scale-invariant case $q \in L^{n/2}$ by Lavine and Nachman (see~\cite{DosSantosFerreiraKenigSalo2013} for details). 

A closely-related problem is Calder\'on's problem, which is to recover the coefficient in the equation $\div (\g \na u) = 0$ from the Dirichlet-to-Neumann map $\Ld_\g$. In~\cite{SylvesterUhlmann1987}, this problem is reduced to the problem of recovering a Schr\"odinger potential $q$, where $q = \g^{-1/2} \lp \g$. Unless $\g$ has two derivatives, the potential $q$ will end up having negative regularity. 

In~\cite{Brown1996,PaivarintaPanchenkoUhlmann2003,BrownTorres2003} it was shown that the Sylvester-Uhlmann argument carries through for conductivities with $3/2$ derivatives. In~\cite{HabermanTataru2013}, the author and Tataru showed uniqueness for $\g \in C^1$ or $\g$ with small Lipschitz norm using an averaging argument. In~\cite{NguyenSpirn2014}, a more involved averaging argument was used to prove uniqueness in three dimensions for $\g \in H^{3/2+}$. In~\cite{Haberman2015a}, the author used arguments similar to those of~\cite{NguyenSpirn2014}, combined with the $L^p$ Carleman estimates of~\cite{KenigRuizSogge1987}, to show uniqueness for $\g \in W^{1,n}$ in dimensions $n=3,4$. This corresponds to recovering a Schr\"odinger potential $q \in W^{-1,n}$.

In two dimensions, the problem has a fairly different character, and we refer the reader to~\cite{SylvesterUhlmann1986,Nachman1996,BrownUhlmann1997,AstalaPaivarinta2006,KhenkinNovikov1987,Novikov1992,Sun1993a,Bukhgeim2008,Blasten2011,ImanuvilovYamamoto2012,GuillarmouSaloTzou2011,Lai2011}.

The main contribution of this paper is in the construction of CGO solutions. These are solutions to the Schr\"odinger equation $L_{A,q} u = 0$ of the form $u = e^{x \cd \z}(a + \psi)$. To construct such solutions, we need to understand the conjugated Laplacian $\lp_\z$, defined by
\[\lp_\z = e^{-x \cd \z} \lp e^{x\cd \z},\]
where $\z \in \C^n$ and $\tau = \a{\Re \z}$ is large. In particular, we would like the show that  the operator $L_{A,q,\z}$, defined by
\begin{equation}
\label{laqz1}
L_{A,q,\z} = e^{-x \cd \z} L_{A,q}e^{x\cd \z} = -\lp_\z - 2i A \cd( \z +\na ) + \dotsb
\end{equation}
is invertible on some function spaces. Thus, we need a lower bound for $\lp_\z$ that can absorb the lower order terms. 

Estimates for operators like $\lp_\z$ arise in the unique continuation problem for operators of the form $-\lp + A \cd \na  + V$. The weak unique continuation property for operators of the form $-\lp + V$, where $V \in L^{n/2}_\loc$, follows from the Carleman estimate of~\cite{KenigRuizSogge1987}
\[\n{e^{-\tau x_1} u}_{p'} \ll \n{e^{-\tau x_1}\lp u}_{p},\]
where $1/p+1/p'=1$ and $1/p-1/p'=2/n$. This estimate is equivalent to an estimate of the form
\[\n{v}_{p'} \ll \n{\lp_{\tau x_1} v}_p\]
for the conjugated Laplacian $\lp_{\tau x_1}$. Similarly, given $A\in L^q_\loc$, unique continuation for the operator $-\lp + A \cd \na$ would follow from a gradient Carleman estimate of the form
\begin{equation}
\label{gradientcarleman}
\n{e^{\tau \phi} \na u}_{L^r} \ll \n{e^{\tau \phi} \lp u}_{L^{p}},
\end{equation}
where $1/p-1/r = 1/q$.

Barcelo, Kenig, Ruiz and Sogge~\cite{BarceloKenigRuizSogge1988} showed that no such gradient Carleman estimate can hold for linear weights of the form $\phi = x_1$ unless $r = p = 2$. In contrast, they proved unique continuation for $A \in L^{(3n-2)/2}_\loc$ and $V \in L^{n/2+}_\loc$ by establishing a gradient Carleman estimate~\eqref{gradientcarleman} for the convex weight $\phi = x_1 + x_1^2$. On the other hand, they showed that for any weight $\phi$, the estimate~\eqref{gradientcarleman} cannot hold uniformly in $\tau$ and $u$ unless the exponents $p$ and $r$ satisfy the condition $1/p-1/r \leq 2/(3n-2)$. This means that the Carleman method cannot be directly applied when $1/q > 2/(3n-2)$.

Nevertheless, it was shown in~\cite{Wolff1992} that the operator $\lp + A \cd \na + q$ has the weak unique continuation property for $A \in L^{n}$ and $V \in L^{n/2}$. A strong unique continuation result of a similar type was proven in~\cite{KochTataru2001}. The idea is that the gradient Carleman estimate~\eqref{gradientcarleman} can be rescued by localizing it to a very small set. Wolff showed that, in a sense, the exponential weight effectively localizes the problem. Unfortunately, it is not clear how this approach could be used to establish uniform bounds for an operator of the form~\eqref{laqz1}, as the argument in~\cite{Wolff1992} only works for a single function $u$.

That localization has a smoothing effect is consistent with the uncertainty principle, since localization in physical space at the scale $\mu^{-1}$ corresponds to averaging in Fourier space at the scale $\mu$. This averaging smooths out the singular behavior of $\lp_\z$ at the characteristic set, so that the distance in Fourier space from the characteristic set $\Char \lp_\z$ is effectively bounded below by $\mu$.  

We take a ``dual'' perspective, by noting that as the modulation $d(\xi,\Char \lp_\z)$ grows, the operator $\lp_\z$ becomes more and more elliptic. This fact holds without localizing in physical space, and so we use this high-modulation gain directly in order to overcome the failure of the gradient Carleman estimate. 

To keep track of the modulation, we use Bourgain-type spaces~\cite{Bourgain1993} with norm $\n{\cd}_{\dot X^{b}_\z}$ given by
\[\n{u}_{\dot X^b_\z} = \n{\lp_\z^b u}_{L^2}.\]
Solvability of~\eqref{laqz1} will follow from a bilinear estimate of the form
\[\a{\j{A\cd (\na + \z) u, v} } \ll \n{u}_{\dot X^{1/2}_\z} \n{v}_{\dot X^{1/2}_\z}.\]
The $\dot X^{1/2}_\z$ norm localizes $\hat u$ and $\hat v$ near the characteristic set $\Si_\z$, which lies in the plane 
\[(\Re\z )^\perp  = \{\xi: \xi \cd \Re(\z) = 0\}\]
 Thus, the worst-case scenario occurs when $\hat A(\xi)$ also concentrates on the plane $(\Re \z)^\perp$.

Using an averaging argument based on Plancherel's theorem, we will show that $\hat A$ cannot concentrate on too many planes through the origin. This will show that $A\cd (\na+\z)$ is a bounded map from $\dot X^{1/2}_\z$ to $\dot X^{-1/2}_\z$ for most values of $\Re \z$. Since $\Re \z$ is, to a large extent, a free parameter, this is enough to run a version of the argument in~\cite{Sun1993}.

We now give an outline of the paper. Sections~\ref{integralidentity}-\ref{theoperator} contain standard material due to~\cite{SylvesterUhlmann1987,Sun1993,EskinRalston1995}. In Section~\ref{dyadic}, we define a dyadic decomposition in frequency and modulation, which will be used extensively throughout the paper. In Section~\ref{xsb}, we introduce the Bourgain spaces $\dot X^b_\z$ and $X^b_\z$ and recall some basic estimates for these spaces from~\cite{HabermanTataru2013,Haberman2015a}. In Section~\ref{averaging}, we review some averaging estimates from~\cite{HabermanTataru2013,Haberman2015a} and prove an additional averaging estimate which follows from the Carleson-Sj\"olin theorem. 

In Section~\ref{amplitude}, we prove new estimates for the amplitude $a$ of the CGO solutions. The amplitude has the form
\[a = \exp(\pb_e^{-1} A),\]
where $\pb_e = (e_1 + i e_2) \cd \na$ for some orthonormal vectors $\{e_1,e_2\}$ in $\R^3$. Since $A$ is only assumed to be in $W^{s,3}$, the amplitude $a$ may behave very badly. Note, however, that if $\pb^{-1}_e$ were replaced by $\a{\na}^{-1}$, then $a$ would be bounded in $L^\infty\cap W^{1,3}_\loc$. By averaging, we show that, for many choices of $e_1$ and $e_2$, the behavior of $a$ is acceptable. In particular, we show that expressions of the form $a q$ (where $q \in W^{-1,3}$) and $\lp a$ are bounded in $X^{-1/2}_\z$. Establishing this is a bit delicate and constitutes the main technical difficulty in this paper relative to previous work. 

We note that it is not too difficult to produce CGO solutions with remainders $\psi$ whose $\dot X^{1/2}_\z$ norm grows like $o(\tau^{1/2})$. This was accomplished in the author's dissertation~\cite{Haberman2015} and is sufficient to determine the magnetic potential $A$. This is because the main term in the integral identity~\eqref{weakmagnetic} has size $\tau$, so errors of order $o(\tau)$ are acceptable. However, once it is shown that the magnetic potentials $A_1$ and $A_2$ coincide, the main term in the integral identity~\eqref{weakmagnetic} has size 1, so the error terms should be of order $o(1)$. If the $\dot X^{1/2}_\z$ norm of the remainders $\psi$ were to grow like $o(\tau^{1/2})$, we would not be able to control the error terms in the integral identity without assuming that $q \in L^\infty$.

In Section~\ref{solvabilityoflaqz}, which contains results from the author's dissertation~\cite{Haberman2015}, we prove estimates for the operator norms of the terms in $L_{A,q,\z}+\lp_\z$. Since $\lp_\z$ maps $\dot X^{1/2}_\z$ isometrically to $\dot X^{-1/2}_\z$, the operator $L_{A,q,\z}$ has a bounded inverse $L_{A,q,\z}^{-1}: \dot X^{-1/2}_\z \to \dot X^{1/2}_\z$ as long as the operator norm $\n{L_{A,q,\z} + \lp_\z}_{\dot X^{1/2}_\z \to \dot X^{-1/2}_\z}$ is sufficiently small. 

In~\cite{Haberman2015a}, the author showed that multiplication by a potential $q$ in $W^{-1,3}$ is bounded in this operator norm by combining the $L^p$ Carleman estimates of~\cite{KenigRuizSogge1987} with an averaging argument. In the present work, we also consider first-order terms such as $A \cd \na$. These terms are more difficult to control, since the behavior of $A \cd \na$ is much worse than the behavior of $q$ when $A$ concentrates at low frequencies. To remedy this, we use the fact that when the frequency of $A$ is sufficiently low, the curvature of the characteristic set does not play an important role. 

In this section of the paper we encounter some logarithmic divergences, which is why we need a regularity assumption on $A$. It is likely that this limitation can be removed, at least for $A \in L^{3+}$, by using a refined version of the pseudodifferential conjugation technique in~\cite{NakamuraUhlmann1994}. This technique should also eliminate the smallness assumption on $A$. We hope to address this problem in future work.

In Section~\ref{proof}, we show that our averaged estimates are sufficient to run the Sylvester-Uhlmann argument. In~\cite{Haberman2015a}, the author concluded the proof of uniqueness in the case $A = 0$ using a compactness argument from~\cite{NguyenSpirn2014}. This argument relies on the decay of the operator norm $\n{q}_{X^{1/2}_\z \to X^{-1/2}_\z}$ as $\tau \to \infty$. It fails for the magnetic Schr\"odinger equation, because the operator norm $\n{A}_{X^{1/2}_\z \to X^{-1/2}_\z}$ does not decay as $\tau \to \infty$, even for smooth $A$. Instead, we use the Banach-Alaoglu theorem to show that the Fourier transforms of $\curl A$ and $q$, which are analytic, vanish on a set of positive measure. 

\section{An integral identity} 
\label{integralidentity}
We now give a very rough outline of how to show that $\Ld_{A_1,q_1}=\Ld_{A_2,q_2}$ implies that $(\curl A_1,q_1)=(\curl A_2,q_2)$ using the Sylvester-Uhlmann strategy. The first step is to write the condition that $\Ld_{A_1,q_1} = \Ld_{A_2,q_2}$  as an integral identity.
\begin{lemma}
If $\Ld_{A_1, q_2} = \Ld_{A_2,q_2}$, then the integral identity
\begin{equation}
\label{weakmagnetic}
\int [i(A_1-A_2)\cd (u_1 \na  u_2 - u_2 \na  u_1) + (A_1^2 -A_2^2 + q_1-q_2)u_1  u_2]\,dx = 0
\end{equation}
holds for any $u_i \in H^1(B)$ solving $L_{A_1,q_1}u_1 = 0$ and $L_{- A_2, q_2} u_2 = 0$ in $B$.
\end{lemma}
\begin{proof}
Define the bilinear form $Q_{A,q}$ by
\begin{align*}
Q_{A,q}(u,v) & = \int_B (\na u \cd \na v + i A \cd(u \na v - v \na u) + (A^2 + q) u v)\,dx.
\end{align*}
If $u$ and $v$ and functions in $H^1(B)$ and $u$ is a weak solution to the equation $L_{A,q}u =0$ in $B$, then
\begin{equation}
\begin{aligned}
0& = \int_{B}( -\div \na u -i\div (Au) - i A\cd \na u + A^2 + q) v\,dx \\
&= Q_{A,q}(u,v) -\int_{\pd B} \pd_\nu u \cd v \,dx,
\end{aligned}
\label{weakmagneticdef}
\end{equation}
where $\nu$ is the outward unit normal to $\pd B$. Thus we have the identity
\begin{equation}
\label{biid}
Q_{A,q}(u,v) = \j{\eval{\pd_\nu u}_{\pd B}, \eval{\bar v}_{\pd \Om}}_{L^2(\pd B)}.
\end{equation}

Suppose we are given functions $u_1$ and $u_2$ in $H^1(B)$ satisfying the equations $L_{A_1,q_1} u_1 =0$ and $L_{-A_2,q_2} u_2 =0$. The assumption that $\Ld_{A_1,q_2} = \Ld_{A_2,q_2}$ implies that there is some $v_2$ in $H^1(B)$ such that $L_{A_2,q_2} v_2 = 0$ and
\begin{align*}
\eval{u_1-v_2}_{\pd \Om} = 0,\qq\eval{\pd_\nu (u_1-v_2)}_{\pd B} =0.
\end{align*}
Thus, by the identity~\eqref{biid}, we derive that
\begin{align*}
Q_{A_1, q_1}(u_1, u_2) & = \j{\eval{\pd_\nu u_1}_{\pd B}, \eval{\bar u_2}_{\pd \Om}}_{L^2(\pd B)}\\
&=\j{\eval{\pd_\nu v_2}_{\pd B}, \eval{\bar u_2}_{\pd \Om}}_{L^2(\pd B)}\\
&= Q_{A_2,q_2}(v_2, u_2) 
\end{align*}
On the other hand, by the definition of $Q_{A,q}(u,v)$, we have $Q_{A,q}(v,u) = Q_{-A,q}(u,v)$. Thus, using the identity~\eqref{biid} again, we derive
\begin{align*}
Q_{A_2,q_2}(v_2, u_2)&=Q_{- A_2,  q_2}( u_2,v_2)\\
&= Q_{-A_2,q_2}(u_2, u_1) \\
&=Q_{A_2,q_2}(u_1,u_2).
\end{align*}
We conclude that $Q_{A_1,q_1}(u_1,u_2) - Q_{A_2,q_2}(u_1,u_2) = 0$, which is~\eqref{weakmagnetic}.

\end{proof}

To use this integral identity, we construct CGO solutions $u_1$ and $u_2$ to the equations $L_{A_1,q_1} u_1 = 0$ and $L_{-A_2,q_2} u_2 = 0$. The CGO solutions $u_i$ are approximately complex exponentials $e^{x\cd \z_i}$, where 
\begin{gather*}
\z_1\sim \tau(e_1 + ie_2)\\
\z_2  \sim -\tau(e_1 + ie_2)\\
\z_1 + \z_2  = ik
\end{gather*}
for some arbitrary vectors $e_1, e_2, k \in \R^3$ satisfying 
\begin{gather*}
e_1 \perp e_2 \perp k\\
\a{e_1}=\a{e_2} =1.
\end{gather*}
Substituting the CGO solutions $u_i \sim e^{x\cd \z_i}$ into the integral identity~\eqref{weakmagnetic} gives
\begin{align*}
0 \sim -2i\tau (e_1 + ie_2) \cd \int_\Om (A_1 - A_2) e^{ik\cd x}\,dx + \int_\Om (A_1^2 - A_2^2 + q_1 - q_2) e^{ik\cd x}\,dx.
\end{align*}
Taking the limit as $\tau \to \infty$, we have
\[0 = (e_1 + ie_2) \cd \widehat{A_1 - A_2}(k)\]
for every pair $\{e_1, e_2\}$ of orthonormal vectors perpendicular to $k$. This implies that $\curl A_1 = \curl A_2$. In particular, by Poincar\'e's lemma, there is a gauge transform $\psi$ such that $A_1 - A_2 = \na \psi$.

The Dirichlet-Neumann relation is invariant under such gauge transforms.
\begin{lemma}
\label{gaugeinvariance}
Suppose $\psi$ is a function supported in $\f 1 2 B$ such that $\psi \in W^{1,3}(B)\cap L^\infty(B)$. Then $\Ld_{A + \na \psi, q} = \Ld_{A, q}$.
\end{lemma}
\begin{proof}
We have
\[e^{-i\psi}L_{A,q} e^{i \psi} = -(\na + iA + i\na \psi)^2 + q.\]
Thus the map $u \mapsto e^{-i \psi} u$ is a bijection between solutions to $L_{A,q} u = 0$ and solutions to $L_{A + \na \psi,q} u =0$. Since $\psi$ is supported in $\f 1 2 B$, multiplication by $e^{i\psi}$ does not change the boundary data, so the conclusion of the lemma follows.
\end{proof}

By the gauge-invariance of the Dirichlet-Neumann relation, we have $\Ld_{A_2,q_2} = \Ld_{A_1,q_2}$. Since we assumed that $\Ld_{A_2,q_2} = \Ld_{A_1,q_1}$, this implies that $\Ld_{A_1,q_1} = \Ld_{A_1,q_2}$. Now construct CGO solutions to the equations $L_{A_1,q_1} u_1 = L_{-A_1,q_2}u_2 = 0$. Substituting the $u_i$ into the integral identity~\eqref{weakmagnetic} again gives
\[0 \sim \int_{\Om} (q_1 - q_2) e^{ik \cd x}\,dx,\]
and we can conclude that $q_1 = q_2$.

\section{A transport equation}
\label{transportequation}
When the magnetic potential $A$ is nonzero, the form of the CGO solutions will depend on $A$. We construct solutions of the form
\begin{equation}
\label{phasesolution}
u = e^{x \cd \z} (a + \psi),
\end{equation}
where $a = e^{-i\phi}$ for a suitable function $\phi$ depending on $\z$. The remainder $\psi$ must solve the equation
\begin{equation}
\label{cgowithphase}
L_{A,q,\z} \psi = -\lp a - 2  \z \cd \na a - i(\na \cd A) a - 2i A \cd \na a - 2 i \z \cd A a + A^2a + qa,
\end{equation}
where the operator $L_{A,q,\z}$ is defined by
\begin{align*}
L_{A,q,\z} & = e^{-x \cd \z} L_{A,q} e^{x \cd \z} \\
& = -(\na + iA+\z)^2 + q.
\end{align*}
In order to eliminate the terms of order $\tau$ on the right hand side of~\eqref{cgowithphase}, we choose $\phi$ so that $a$ solves (roughly speaking) a transport equation of the form
\[ \z \cd \na a = -i \z \cd A a.\]
Equivalently, the function $\phi$ satisfies an equation of the form
\[\z \cd \na \phi = \z\cd A .\]
Since $\z = \tau (e_1 + i e_2)$, where $e_1$ and $e_2$ are orthonormal vectors, this a $\pb$ equation for $\phi$ in the plane determined by $e_1$ and $e_2$. 

Given $e=e_1 + i e_2$, where $e_1$ and $e_2$ are orthogonal unit vectors, define
\[\pb_e = e_1 \cd \na + i e_2 \cd \na.\]
We now assume for simplicity that $e_1$ and $e_2$ are the standard basis vectors. In this case, the operator $\pb$ is given by
\[\pb = \pd_1 + i \pd_2.\]
Let $f$ be a function defined on the complex plane, which we identify with $\R^2$ by writing $z = z_1 + iz_2$. The equation
\[\pb u = f\]
is of Cauchy-Riemann type. If $f$ is smooth and compactly supported, then it has a solution given by the formula
\[\pb^{-1} f(w) = \frac{1}{2\pi} \int\frac{f(w-z)}{z}\,dz_1\,dz_2.\]
The kernel $(2 \pi z)^{-1}$ is locally integrable, so it has good mapping properties.
\begin{lemma}
\label{dbli}
If $f: \C \to \R$ is supported in $B(0,1/2)$, then
\[\n{\j{w}\pb_e^{-1}f(w)}_{L^\infty} \ll \n{f}_{L^\infty}.\]
\end{lemma}
\begin{proof}
Write
\begin{equation*}
\a{\pb_e^{-1} f(w)} \ll \n{f}_{L^\infty} \int \chi_{B(0,1/2)}(z-w) \a{z}^{-1}\,dz_1\,dz_2
\end{equation*}
When $\a{w} \leq 1$, we estimate the integral by 
\[\int_{B(0,3/2)} \a{z}^{-1}\,dz_1\,dz_2 \sim 1.\]
When $\a{w} > 1$, we have $\a{z} \geq \a{w}/2$ in the region of integration, so we estimate instead by
\[\a{w}^{-1} \int \chi_{B(0,1/2)}(z-w) dz_1\,dz_2 \sim \a{w}^{-1}.\]
\end{proof}

When we substitute CGO solutions of the form $u_i = e^{x \cd \z_i} (e^{i \phi_i} + \psi)$ into the integral identity~\eqref{weakmagnetic}, the main term has the form
\[-i(\z_1 - \z_2) \cd \int_\Om (A_1 - A_2) e^{i(\phi_1 - \phi_2)} e^{ix\cd k}\,dx.\]
The next lemma, due to~\cite{EskinRalston1995}  says that we can remove the factor $e^{i (\phi_1-\phi_2)}$ from this integral and recover the Fourier transform.
\begin{lemma}
\label{undophase}
Let $e_1, e_2, k \in \R^n$ be arbitrary vectors satisfying $\a{e_1}=\a{e_2}=1$ and $e_1\cd e_2 = e_1\cd k= e_2\cd k=0$. Let $A \in C_0^\infty(\R^n)$, and let $\phi =  \pb_e^{-1}(e\cd  A)$. Then
\[(e_1+ie_2) \cd \int A e^{-i\phi} e^{ix\cd k} \,dx = (e_1 + ie_2)\cd \int A e^{i x\cd k} \,dx.\]
\end{lemma}
\begin{proof}
We assume that $e_1$ and $e_2$ are the first two standard basis vectors. Since
\[(e_1+ie_2) \cd A e^{-i \phi} = i \pb_e (e^{-i \phi}),\]
and $k = (0,0,k')$, we may write
\begin{align}\label{fub}
(e_1 + ie_2) \cd \int A e^{-i \phi(x)} e^{i x\cd k}\,dx & = i\int \pb_e(e^{-i \phi(x_1,x_2,x')}) e^{ ix'\cd k'}\,dx_1\,dx_2\,dx'.
\end{align}
By the divergence theorem, we have
\begin{align}
\label{thisthing}
\int (\pd_1 + i \pd_2)(e^{-i \phi(x_1,x_2,x')})\,dx_1\,dx_2  & = \lim_{R \to \infty} \int_{\pd B(0,R)}(\nu_1 + i \nu_2) e^{-i \phi(x_1,x_2,x')}\,dS,
\end{align}
where $\nu$ is the outward unit normal on the circle $\pd B(0,R)$. By Lemma~\ref{dbli}, the estimate $\a{\phi} = O(1/\j{x_1 + ix_2})$ holds uniformly in $x'$. Thus we have a Taylor expansion of the form
\[e^{-i \phi} = 1 - i \phi + O(\j{x_1 + ix_2}^{-2}).\]
Substituting the Taylor series into the right hand side of~\eqref{thisthing}, we find that
\begin{align*}
\int_{\pd B(0,R)}(\nu_1 + i \nu_2) e^{i \phi(x_1,x_2,x')}\,dS & = \int_{\pd B(0,R)}(\nu_1 + i \nu_2) \,dS - i\int_{\pd B(0,R)} (\nu_1 + i \nu_2)  \phi\,dS \\
&\qq+ \int_{\pd B(0,R)} O(R^{-2}) \,dS\\
&=\int_{B(0,R)} \pb_e(1) \,dx_1 \,dx_2 -i \int_{B(0,R)}  \pb_e \phi\,dx_1\,dx_2 + O(R^{-1})
\end{align*}
Taking the limit as $R \to \infty$ we obtain the identity
\begin{align*}
\int (\pd_1 + i \pd_2)(e^{i \phi(x_1,x_2,x')})\,dx_1\,dx_2 & = -i\int \pb_e \phi\,dx_1\,dx_2 \\
&=-i \int e \cd A\,dx_1\,dx_2
\end{align*}
Substituting the identity into~\eqref{fub}, this proves the lemma.
\end{proof}
\section{The operator $\lp_\z$}
\label{theoperator}
In order to construct solutions to the equation~\eqref{cgowithphase} for the remainder $\psi$, we consider operators of the form
\[\lp_\z = e^{-x\cd \z} \lp e^{x\cd \z}.\]
The complex vector $\z \in \C^3$ is given by
\[\z = \tau (e_1 + i \eta),\]
where $\tau > 0$, $\a{e_1} = 1$, $\a{\eta} \leq 1$ and $\eta \perp e_1$. The symbol of $\lp_\z$ is
\[p_\z(\xi) = (i\xi + \z)^2 = -(\xi+\tau \eta)^2 +2 i\tau e_1 \cd \xi + \tau^2.\]
The characteristic set $\Si_\z$ is the intersection of the plane perpendicular to $e_1$ and a sphere centered at $-\tau \eta$.
\[\Si_\z = \{\xi: \xi\cd e_1 = 0, \a{\xi + \tau \eta} = \tau\}.\]
We will refer to the distance from this set as the {\em modulation}. The symbol $p_\z$ is elliptic at high modulation and vanishes simply on $\Si_\z$.
\begin{equation}
\label{pbehavior}
\a{p_\z} \sim \begin{cases}
\tau d(\xi, \Si_\z) &\WHEN d(\xi,\Si_\z) \leq \tau/8\\
\tau^2 + \a \xi^2 &\WHEN d(\xi,\Si_\z) \geq \tau/8
\end{cases}
\end{equation}

\section{Dyadic projections}
\label{dyadic}
If $m$ is a smooth function on $\R^n$, then $m(D)$ will denote the Fourier multiplier with symbol $m(\xi)$.

Let $\chi \in C^\infty_0([0,1])$ be a smooth function such that $\chi = 1$ on $[0,3/4]$. For each dyadic integer $\ld = 2^k$,  define the Littlewood-Paley projection $P_{\leq \ld}$ onto frequencies of magnitude $\a{\xi} \leq \ld$ by $P_{\leq \ld} = \chi(\a{D}/\ld)$. Similarly, define the projection $P_{>\ld}$ onto frequencies of magnitude $\a{\xi} \gg \ld$ by $P_{>\ld} = I-P_{\leq \ld}$, and define the projection $P_{\ld}$ onto frequencies of magnitude $\a{\xi} \sim \ld$ by $P_\ld = P_{\leq \ld} - P_{\leq \ld/2}$. 

Note that $I = \sum_{\ld} P_\ld$. Thus we can decompose a function $f$ into a sum of dyadic pieces $f_\ld = P_\ld f$.

We can use the Littlewood-Paley decomposition to characterize the Besov spaces $B^s_{p,q}$. Given $s \in \R$, $p \in (1,\infty)$, and $q \in [1,\infty]$, the Besov space $B^s_{p,q}$ is characterized by the norm
\[\n{u}_{B^s_{p,q}} = \n{u_{\leq 1}}_p + \sps{\sum_{\ld > 1} \n{u_\ld}_p^q}^{1/q}.\]
For any integer $k$, the Littlewood-Paley square function estimate implies that $B^k_{p,2} \subset W^{k,p} \subset B^k_{p,p}$ when $p\geq 2$ and that $B^k_{p,p} \subset W^{k,p} \subset B^k_{p,2}$ when $p \leq 2$. When $s$ is not an integer, the Sobolev space $W^{s,p}$ is usually defined in such a way that $W^{s,p} = B^s_{p,p}$.

Given a pair $\{e_1,e_2\}$ of orthonormal vectors, we set $e = e_1 + i e_2$ and define partial Littlewood-Paley projections 
\begin{align*}
P^{e_1}_{\leq \ld} & = \chi (\a{D \cd e_1}/\ld) \\
P^{e}_{\leq \ld} & =\chi(\a{D \cd e}/\ld).
\end{align*}
We define $P^{e_1}_\ld$ and $P^e_{\ld}$ in a similar way.

Next, we define projections $Q_\nu^\z$ to regions where $d(\xi,\Si_\z) \sim \nu$. Let $\z$ be a complex vector of the form $\z = \tau(e_1 + i \eta)$. We first define the projection $C_{\leq \nu}^\z$ by
\[C_{\leq \nu}^\z = \chi((\a{D^\perp + \tau e_2} - \tau)/\nu),\]
where $\xi^\perp = \xi - (\xi\cd e_1) e_1$. We then define the projection $Q_{\leq \nu}^\z$ by
\[Q_{\leq \nu}^\z = P^{e_1}_{\leq \nu} C_{\leq \nu}^\z.\]
Finally, we define the projection $Q_\nu^\z$ by $Q_\nu^\z = Q_{\leq \nu}^\z - Q_{\leq \nu/2}^\z$ as before.

When the choice of $\z$ is clear from context, we will suppress the dependence of the $Q$ projections on $\z$. Similarly, we will write $P^1$ instead of $P^{e_1}$ or $P^{Ue_1}$.

Define projections $Q_l$ and $Q_h$ onto low and high modulation by 
\begin{align*}
Q_l^\z & = Q_{\leq \tau/8}^\z \\
Q_h^\z & = Q_{> \tau/8}^\z.
\end{align*}
Note that the projection $Q_h^\z$ projects to the region where $\lp_\z$ is elliptic. 

\section{The $X_\z^{b}$ spaces}
\label{xsb}
Given $b \in (-1,1)$, define the homogeneous $\n{\cd}_{\dot X^{b}_\z}$ norm by
\[\n{u}_{\dot X^b_\z} = \n{\lp_\z^b u}_2,\]
and define the inhomogeneous $\n{\cd}_{X^b_\z}$ norm by
\[\n{u}_{X^b_\z} = \n{(\a{\lp_\z} + \tau)^b u}_2.\]
By the symbol estimates~\eqref{pbehavior}, we have the low-modulation $L^2$ estimate
\begin{align}
\label{l2estimate}
\n{Q_\mu u}_{X^b_\z}  & \sim (\mu \tau)^b \n{Q_\mu u}_2,
\end{align}
which holds for $\mu \leq \tau/8$, and the high-modulation $L^2$ estimate
\begin{align}
\label{hmestimate} 
\n{Q_h u}_{X^b_\z} & \sim \n{u}_{H^b_\tau},
\end{align} 
where the semiclassical $\n{\cd}_{H^b_\tau}$ norm is defined by
\[\n{u}_{H^b_\tau} = \n{(\a{D}+\tau)^b u}_2.\]
The $X^b_\z$ spaces behave well under localization, as we see from the following lemma.
\begin{lemma}[\cite{HabermanTataru2013}]
\label{loclemma}
If $\phi$ is a Schwartz function, then
\begin{align}
\label{loc}
\n{\phi u}_{X^{1/2}_\z} & \ll_\phi \n{u}_{\dot X^{1/2}_\z}\\
\label{locdual}
\n{\phi u}_{\dot X^{-1/2}_\z} &\ll_\phi \n{u}_{X^{-1/2}_\z}.
\end{align}
\end{lemma}
By a scaling argument, Lemma~\ref{loclemma} implies the Agmon-H\"ormander-type estimate
\begin{equation}
\label{agmonhormander}
\tau^{1/2} R^{-1/2} \n{u}_{L^2(B(0,R))} \ll \n{u}_{\dot X^{1/2}_\z}.
\end{equation}
Using the estimate~\eqref{agmonhormander}, it is not hard to show that the space $\dot X^{1/2}_\z$ is a Banach space and embeds continuously into $\tau^{-1/2} L^2(\R^3, \j x^{-1-\d}\,dx)$ for any $\d>0$.

The next lemma gives Strichartz-type estimates for the $X^b_\z$ spaces.
\begin{lemma}[\cite{KenigRuizSogge1987,Haberman2015a}]
Suppose $\nu\leq \tau/8$. Then for any $f \in \dot X^{1/2}_\z(\R^3)$, we have
\begin{align}
\label{bandstrichartz}
\n{Q_\nu f}_{6} &\ll (\nu/\tau)^{1/3} \n{f}_{\dot X^{1/2}_\zeta}.\\
\label{strichartz}
\n{f}_6& \ll \n{f}_{\dot X_\zeta^{1/2}}.
\end{align}
\end{lemma}
We also have the dual estimates
\begin{align}
\label{dualbandstrichartz}
\n{Q_\nu f}_{X^{-1/2}_\z} & \ll (\nu/\tau)^{1/3} \n{f}_{6/5}\\
\label{dualstrichartz}
\n{f}_{X^{-1/2}_\z} & \ll \n{f}_{6/5}.
\end{align}

\section{Averaging estimates}
\label{averaging}
We will need to average various norms with respect to parameters $(\tau, U)$, which will be chosen from the set $[2,\infty) \times O(3)$. In order to distinguish this averaging from integration over physical space, we will use probabilistic notation.

Let $(X,\si)$ be a finite measure space such that $\si(X)>0$. Let $Z$ be an integrable function on $X$. We write the average of $Z$ over $X$ as
\begin{align*}
\E[Z \mid X] =\si(X)^{-1} \int_X Z\,d\si.
\end{align*}
Similarly, for a measurable subset $Y$ of $X$, we write
\begin{align*}
\Pr[Y\mid X] & = \frac{\si(Y\cap X)}{\si(X)}.
\end{align*}
Define the $L^p$ average of $Z$ over $X$ by
\[\E^p[Z\mid X] = \frac{ \n{Z}_{L^p(X,d\si)}}{\n{1}_{L^p(X,d\si)}}.\]
Unless otherwise specified, the set $X$ will be the orthogonal group $O(3)$, and $\si$ will be normalized Haar measure on $O(3)$. 

Given a measurable function $Z$ on $[\tau_*,2\tau_*] \times O(3)$, define 
\[\E^p_{\tau_*}[Z] = \E^p[Z\mid [\tau_*,2\tau_*] \times O(3)],\]
where the average here is taken with respect to the measure $m$ on $[2, \infty) \times O(3)$ given by
\[dm(\tau,U) = (\tau \log \tau)^{-1} \,d\tau\,d\si(U).\]
For a positive integer $K$, define
\[\ti\E^p_K[Z] = \E^p[Z\mid [2^K,2^{K^2}] \times O(3)]\]
and
\[\ti \Pr_K[Y] = \Pr[Y\mid [2^K,2^{K^2}] \times O(3)],\]
Note that the quantity $m([2^K,2^{K^2}] \times O(3))$ is given by 
\begin{equation}
\label{logk}
\int_{2^K}^{2^{K^2}} (\tau \log \tau)^{-1} \,d\tau \,d\si(U) \sim \log K.
\end{equation}
On each dyadic interval $[\tau^*,2\tau^*]$, the weight $(\tau \log \tau)^{-1}$ is approximately constant. Thus we can estimate $\ti \E_k^p[Z]$ by
\[\ti \E_K^p[Z]^p \sim (\log K)^{-1} \sum_{2^K \leq \tau_* < 2^{K^2}} (\log \tau_*)^{-1} \E_{\tau_*}^p[Z]^p.\]

We will need the following property of the Haar measure: if $f$ is an integrable function on $S^{2}$, then for any fixed $\th\in S^{2}$ we have the identity
\begin{equation}
\label{haarformula}
\E[f(U\cd \th) \mid O(3)]= \E[f(\om)\mid S^2],
\end{equation}

We will use the averaging in $\tau$ to take advantage of the extra decay in expressions of the form $(\ld /\tau)^\al \n{P_\ld f}_p$, where $\ld \ll \tau$. Namely, if $\al > 0$ and $p \in [2,\infty)$, then we have frequency convolution estimate
\begin{equation}
\label{loggain}
\ti\E_K^p\left[\sum_{\ld \ll \tau} (\ld/\tau)^{\al} (\log \tau)^{1/p} \n{P_\ld f}_p\right] \ll  (\log K)^{-1/p} \n{f}_p.
\end{equation}
To see this, we recall the normalization~\eqref{logk} and use Young's inequality, which gives
\begin{align*}
\ns{\sum_{\ld \ll \tau} (\ld/\tau)^{\al} \n{P_\ld f}_p}_{L^p([2^K,2^{K^2}], \tau^{-1} \,d\tau)} & \ll \sps{\sum_\ld \n{P_\ld f}^p_p}^{1/p} .
\end{align*}
By the Besov embedding $L^p \subset B^0_{p,p}$, the right hand side is bounded by $\n{f}_p$.
\begin{lemma}
\label{avgestimates}
Let $p \in [2,\infty]$ and let $1/p'=1-1/p$. Let $\mu,\nu,\ld$ be dyadic integers, such that $\mu,\nu\ll \ld$. Then
\begin{align}
\label{pavg}
\E^p[\j{\ld/\nu}^{1/p}\j{\ld/\mu}^{1/p} \n{P^{Ue_1}_{\leq \mu} P^{Ue_2}_{\leq \nu} P_\ld u}_p] &\ll \n{ u}_p\\
\label{pavgdual}
\E^{p'}[\j{\ld/\nu}^{1/p}\j{\ld/\mu}^{1/p} \n{P^{Ue_1}_{\leq \mu} P^{Ue_2}_{\leq \nu} P_\ld u}_{p'}]& \ll \n{ u}_{p'}
\end{align}
If $p \in [2,4]$, then we also have
\begin{equation}
\label{qavg}
\E^p_{\tau*}[(1 + \log_+ (\tau/\nu))^{5(1/p-1/2)} \j{\ld/\nu}^{3/p-1/2}\n{Q_{\leq \nu}^{\tau(Ue_1+iUe_2)} P_\ld u}_p]  \ll \n{u}_p
\end{equation}
\end{lemma}
\begin{proof} 
When $p=2$, all of these estimates follow from Plancherel's theorem and Fubini. To prove the first estimate~\eqref{pavg} when $p=2$, we write
\begin{align*}
\E^2[\n{P^{Ue_1}_{\leq \nu} P^{Ue_2}_{\leq \mu}P_\ld u}_2]^2 & \sim \int_{\R^n} \E\as{\phi\sps{\frac \xi \ld} \chi\sps{\frac{\xi\cd (Ue_1)}{\mu},\frac{\xi \cd( Ue_2)}{\nu}}}^2 \a{\hat u(\xi) }^2\,d\xi
\end{align*}
Here $\phi$ is supported on an annulus, and $\chi$ is supported on a square. Since $U$ is orthogonal, we have $\xi \cd (Ue_i) = (U^{-1} \xi) \cd e_i$. Thus we can compute the last integral using the identity~\eqref{haarformula}:
\begin{align*}
\E\as{\phi\sps{\frac \xi \ld} \chi\sps{\frac{(U^{-1}\xi)\cd e_1}{\mu},\frac{(U^{-1}\xi) \cd e_2}{\nu}}}^2&\ll \sup_{\a{\xi} \sim \ld} \E\as{\chi\sps{\frac{\a{\xi} \om \cd e_1}{\mu},\frac{\a{\xi} \om \cd e_2}{\nu}}}^2.
\end{align*}
The quantity on the right is bounded by the area of the intersection of the unit sphere with a rectangle centered at the origin of size proportional to $(\mu/\ld) \times (\nu/\ld)\times 1$. Since the area of such a region is bounded by $\j{\ld/\mu}^{-1} \j{\ld/\nu}^{-1}$, we have
\[\E^2[\n{P^{Ue_1}_{\leq \nu} P^{Ue_2}_{\leq \mu}P_\ld u}_2]^2\ll  \j{\ld/\mu}^{-1} \j{\ld/\nu}^{-1} \n{u}_2^2.\]
The $p=2$ case of the last estimate~\eqref{qavg} is proven in the same way. Since $\a{p_\z(\xi)} \ll \tau \nu$ on the Fourier support of $Q_\nu$, it suffices to show that 
\begin{equation}
\label{etsgoal}
\E_{\tau_*}[Z(\tau,U)] \ll \j{\ld/\nu}^{-1},
\end{equation}
where 
\begin{align*}
Z(\tau,U) &= \sup_{\a\xi \sim \ld} \as{\chi\sps{\frac{\a{-\a\xi^2 - 2 \tau \xi\cd (U e_2)} + 2\tau \a{\xi \cd (U e_1)}}{\tau\nu}}}^2
\end{align*}
and $\chi$ is compactly supported. Using the identity~\eqref{haarformula} again yields
\begin{align}
\label{etssup}
\E_{\tau_*} [Z] & \ll \sup_{\a{\xi}\sim\ld} \frac{1}{\tau_*} \int_{\tau_*}^{2\tau_*}\int_{S^2}\as{\chi\sps{\frac{\a{-\a\xi^2 - 2 \tau \a \xi \om \cd e_2} + 2\tau \a{\a \xi \om \cd e_1}}{\tau\nu}}}^2 \,dS(\om)\,d\tau .
\end{align}
View $(\tau,\om)$ as polar coordinates on $\R^3$, and change variables to $u = \tau\om$. In the annular region $\{\a u \in [\tau_*,2\tau_*]\}$, the volume element $du$ is bounded below by $\tau_*^{2} \, dS(\om) \,d\tau$. Thus the integral on the right hand side of~\eqref{etssup} is bounded by
\[\frac{1}{\tau_*^3} \int_{\a u \in [\tau_*,2\tau_*]}\as{\chi\sps{\frac{\a{\xi}( \a{-\a\xi - 2 u \cd e_2} + 2 \a{ u \cd e_1})}{\tau\nu}}}^2 \,du.\]
The integrand is supported on a rectangle of size proportional to $\tau(\j{\ld/\nu}^{-1} \times \j{\ld/\nu}^{-1} \times 1)$. So the integral is bounded by the quantity $\j{\ld/\nu}^{-2}$, which establishes~\eqref{etsgoal}. This shows that
\[\E^2_{\tau^*}[\n{Q_{\leq \nu} P_\ld u}_2]^2 \ll \j{\ld/\nu}^2 \n{u}_2^2.\]
To prove the $p \neq 2$ case of the first two estimates~\eqref{pavg} and~\eqref{pavgdual}, we define an operator $T$ by
\begin{align*}
Tu(U,x) & := P^{U e_1}_{\leq \mu} P^{U e_2}_{\leq \nu} P_\ld u(x).
\end{align*}
We have shown that $T$ satisfies the $L^2$ bound
\[\n{T}_{L^2(\R^3)\to L^2(O(3) \times \R^3)} \ll \j{\ld/\nu}^{-1/2} \j{\ld/\mu}^{-1/2}.\]
On the other hand, the operator $T$ is built out of operators that are bounded on $L^\infty$ and $L^1$. Thus $T$ also satisfies the bounds
\[ \n{T}_{L^\infty(\R^3) \times L^\infty(O(3)\times \R^3)}+ \n{T}_{L^1(\R^3)\to L^1(O(3) \times \R^3)}\ll 1.\]
By interpolation, we obtain the $L^p$ bounds~\eqref{pavg} and~\eqref{pavgdual}.

To prove the $p\neq 2$ case of the last estimate~\eqref{qavg} for $Q_{\leq \nu}^{\z(\tau,U)}$, we interpolate with an $L^4$ bound. The operator $Q_{\leq \nu}$ factors as
\[Q^{\z(\tau,U)}_{\leq \nu} = C^{\z(\tau,U)}_{\leq \nu} P^{Ue_1}_{\leq \nu},\]
where the operator $C_{\leq \nu}^{\z(\tau,U)}$, defined by
\[C_{\leq \nu}^{\z(\tau,U)} = \chi((\a{D^\perp + \tau e_2} - \tau)/\nu),\]
localizes the vector $\xi^\perp = (0,\xi \cd U e_2, \xi\cd Ue_3)$ to a neighborhood of a circle of radius $\tau$ and center $(\xi\cd (Ue_1), \tau e_2)$. The Carleson-Sj\"olin theorem (\cite{CarlesonSjolin1972}) implies that $C^{\z(\tau,U)}_{\leq \nu}$ satisfies the $L^4$ bound
\begin{equation}
\label{carlesonsjolin}
\n{C_{\leq \nu}^{\z(\tau,U)}}_{L^4(\R^3) \to L^4(\R^3)} \ll (1 + \log_+(\tau/\nu))^{5/4}.
\end{equation}
This estimate (modulo rescaling and modulation) is explicit in~\cite{Cordoba1977}. Thus, by combining the $L^4$ bound~\eqref{carlesonsjolin} with the case $p=4$ of the bound~\eqref{pavg} that we have already established, we obtain the $L^4$ bound
\begin{align*}
\E^4_{\tau^*}[\n{Q_{\leq \nu}^{\z(\tau,U)} u}_4]&\ll(1 + \log_+(\tau/\nu))^{5/4} \E^4_{\tau^*}[\n{P^{U e_1}_{\ll \nu} u}_4]\\
&\ll (1 + \log_+(\tau/\nu))^{5/4} \j{\ld/\nu}^{-1/4} \n{u}_4.
\end{align*}
Interpolating this $L^4$ estimate with the $L^2$ estimate we have already established, we obtain the $L^p$ estimate~\eqref{qavg}.
\end{proof}

\section{Estimates for the amplitude}
\label{amplitude}
To analyze the behavior of $\pb^{-1}$, we introduce an auxiliary function $\eta$ to use as a mollifier.

Let $\eta: \C \to \R$ be a smooth compactly supported bump function, such that $\int_\C \eta = 1$ and
\begin{align}
\label{moments}
\int z_1^{\al_1}z_2^{\al_2} \eta(z) \,dz_1 \,dz_2 = 0,
\end{align}
for every $\al = (\al_1,\al_2)$ such that $1 \leq \al_1 + \al_2 \leq 2M$, where $M$ is some large number to be determined later. The vanishing moment condition~\eqref{moments} ensures that $\hat \eta$ satisfies 
\begin{equation}
\label{momentestimate}
\hat \eta(\xi) = 1 + O(\a{\xi}^{2M+1}).
\end{equation}
Let $e = e_1 + i e_2$, and define the function $\eta^e$ as before by $\eta^e(x) = \eta(x\cd e)$. Define the operator $\ti P^e$ by
\[\ti P^e u = \eta^e * u. \]
Let $\chi^e_\nu$ be the symbol of the Littlewood-Paley projection $P^e_{\nu}$. Since the inverse Fourier transform $\check \eta^e$ is a Schwartz function, we have
\[\a{\na^k_\xi (\chi^e_\nu \check \eta^e)(\xi)} \ll_{k,N} \nu^{-N}\]
for all nonnegative integers $k , l, N$ and uniformly in $\nu \geq 1$. It follows that, for any $p \in [1,\infty]$ and $N > 0$, the product $\ti P^e P^e_\nu$ satisfies the almost orthogonality bound
\begin{equation}
\n{\ti P^e P^e_\nu}_{L^p \to L^p} \ll_N \nu^{-N}
\label{hfhm}
\end{equation}
for all $N> 0$, and uniformly in $\nu \geq 1$. On the other hand, the vanishing property~\eqref{momentestimate} implies that
\[\a{\na_\xi^k (\chi^e_\nu(1- \check \eta^e))(\xi)} \ll_{k} \nu^{2M+1-k}\]
for all nonnegative integers $k \leq M$ and uniformly in $\nu \leq 1$. It follows that for $M$ sufficiently large, the product $(1-\ti P^e) P_\nu^e$ satisfies the almost orthogonality bound
\begin{equation}
\n{(1-\ti P^e)P^e_\nu}_{L^p \to L^p} \ll \nu^M \n{u}_{L^p}
\label{lfhm}
\end{equation}
for all $p \in [1,\infty]$ and uniformly in $\nu \leq 1$.

We apply this machinery to show that the behavior of $\pb$ near its the characteristic set can be ignored if everything is localized. Since the next lemma pertains to functions of two variables, we write $\ti P$ and $\pb$ instead of $\ti P^e$ and $\pb_e$.
\begin{lemma}
Let $p \in [1,\infty]$. If $u$ is a function on $\R^2$  supported in the unit ball $B = B(0,1)$, then
\begin{equation}
\n{\pb^{-1}u}_{\infty} \ll \n{P_{\leq 1} u}_\infty + \n{\pb^{-1} P_{>1} u}_\infty
\label{dbiloc}
\end{equation}
and
\begin{equation}
\n{\pb^{-1}u}_{L^p(B)} \ll \n{P_{\leq 1} u}_{p} +  \n{\pb^{-1} P_{>1} u}_{p}
\label{dbploc}
\end{equation}
\end{lemma}
\begin{proof}
Decompose $u$ as $u = \ti P u + (1-\ti P) u$. Since $\ti P u$ is supported in $B(0,2)$, we have $\n{\pb^{-1} \ti P u}_{L^p(B)} \ll \n{\pb^{-1} \ti Pu }_\infty$. By Lemma~\ref{dbli} and Bernstein's inequality
\begin{align*}
\n{\pb^{-1} \ti P u}_\infty & \ll \n{\ti P u}_\infty\\
&\ll \n{\ti P P_{\leq 1} u}_\infty + \sum_{\nu > 1} \n{\ti P P_\nu u}_\infty \\
& \ll \n{P_{\leq 1} u}_p + \sum_{\nu > 1} \nu^{2/p} \n{\ti P P_\nu u}_p
\end{align*}
By the almost orthogonality bound~\eqref{hfhm}, we have
\begin{align*}
\sum_{\nu > 1} \nu^{2/p} \n{\ti P P_\nu u}_p & \ll \n{P_{\leq 1} u}_p + \sum_{\nu >1} \nu^{2/p - N} \n{P_\nu u}_p\\
&\ll \n{P_{\leq 1} u}_p + \sum_{\nu >1} \nu^{2/p+1 - N} \n{P_\nu \pb^{-1} u}_p\\
&\ll \n{P_{\leq 1} u}_p + \n{\pb^{-1} P_{>1} u}_p.
\end{align*}

For $(1-\ti P)u$ we use the almost orthogonality bound~\eqref{lfhm} and the fact that $(1-\ti P)$ is bounded on $L^p$ for any $p$. Thus
\begin{align*}
\n{\pb^{-1} (1-\ti P) u}_p & \ll \sum_{\nu \leq 1} \nu^{-1} \n{(1- \ti P) P_\nu u}_p + \n{\pb^{-1} P_{>1} u}_p\\
&\ll \sum_{\nu \leq 1} \nu^{M-1}  \n{P_{\leq 1} u}_p + \n{\pb^{-1} P_{>1} u}_p\\
&\ll \n{P_{\leq 1} u}_p + \n{\pb^{-1} P_{>1} u}_p.
\end{align*}
\end{proof}
Using the localization estimate~\eqref{dbploc}, we show that $\pb_{Ue}^{-1} \na f$ is bounded on average in $L^2(B)$ with a slight loss of regularity.
\begin{lemma}
\label{h1}
Let $f \in H^s(\R^3)$, where $s>0$, and suppose that $\supp f \subset B$, where $B = B(0,1)$. Then
\[\E^2[\n{\pb^{-1}_{Ue} \na f}_{L^2(B)}] \ll_s \n{f}_{H^s}.\]
\end{lemma}
\begin{proof}
First we apply the localization estimate~\eqref{dbploc}, to obtain
\begin{align*}
\n{\pb^{-1}_{Ue} \na f}_{L^2(B)} & \ll \n{P^{Ue}_{\leq 1} \na f}_{L^2} + \n{\pb_{Ue}^{-1} \na f}_{L^2}.
\end{align*}
We bound both terms using the averaging estimate~\eqref{pavg}.
\begin{align*}
\E^2[\n{P^{Ue}_{\leq 1} \na f}_{L^2}] & \ll \n{\na f_{\leq 1}}_2 + \sum_{\ld>1} \ld \E^2[\n{P_{\leq 1}^{Ue} f_\ld}_2] \\
&\ll \n{f}_2 + \sum_{\ld > 1} \n{f_\ld}_2 \\
&\ll \n{f}_{H^s}.
\end{align*}
Similarly, since $P^{Ue}_\nu f_\ld = 0$ unless $\nu \ll \ld$, we have
\begin{align*}
\E^2[\n{P^{Ue}_{>1} \pb_{Ue}^{-1} \na f}_{L^2}] & \ll \sum_{1 \leq \nu\ll \ld} (\ld/\nu) \E^2[\n{P_{\nu}^{Ue} f_\ld}_2] \\
&\ll \sum_{\ld > 1} (\log \ld) \n{f_\ld}_2 \\
&\ll \n{f}_{H^s}.
\end{align*}
\end{proof}

We now show that the $\pb^{-1}_{U e}$ operator takes compactly supported functions in the Besov space $B^0_{3,1}(\R^3)$ to bounded functions. If the $\pb^{-1}_{Ue}$ operator were replaced by $\a{D}^{-1}$, then this property would hold without any averaging.
\begin{lemma}
\label{phili}
Let $f \in B^0_{3,1}(\R^3)$, and suppose that $\supp f \subset B(0,1)$. Then
\[\E^3[\n{\pb^{-1}_{Ue} f}_\infty] \ll \n{f}_{B^0_{3,1}}.\]
\end{lemma}
\begin{proof}
First, we show that for such $f$, we have the estimate
\begin{equation*}
\n{\pb^{-1}_{Ue} f}_\infty \ll Z(U),
\end{equation*}
where
\begin{equation*}
Z (U)= \n{f}_3 + \sum_{1 \leq \nu, \ld} (\ld/\nu)^{1/3} \n{P^{Ue}_{\leq \nu} P_\ld f}_3.
\end{equation*}
To this end, we apply the localization estimate~\eqref{dbiloc}, which gives
\begin{align*}
\n{\pb^{-1}_{Ue} f}_\infty & \ll \n{P^{Ue}_{\leq 1} f}_\infty +\n{\pb^{-1}_{Ue} P^{Ue}_{>1}f}_\infty.
\end{align*}
Next, we decompose $f$ into Littlewood-Paley pieces and apply Bernstein's inequality. We estimate $P^{Ue}_{\leq 1} f$ by
\begin{align*}
\n{P^{Ue}_{\leq 1} f}_\infty & \ll \n{P^{Ue}_{\leq 1} f_{\leq 1}}_\infty +\sum_{\nu,\ld >1} \n{P^{Ue}_{\leq 1} f_\ld}_\infty \\
&\ll \n{f}_3  +\sum_{\nu,\ld >1} \ld^{1/3} \n{P^{Ue}_{\leq 1} f_\ld}_3.\\
&\leq Z(U).
\end{align*}
We estimate $\pb^{-1}_{Ue}P^{Ue}_{>1} f$ in the same way. Note that $P_\nu^{Ue} f_\ld$ vanishes unless $\nu \ll \ld$, so
\begin{align*}
\n{\pb^{-1}_{Ue} P^{Ue}_{>1}f}_\infty & \ll \sum_{1<\nu \ll \ld} \nu^{-1} \n{P^{Ue}_\nu f_\ld}_\infty \\
&\ll \sum_{1<\nu \ll \ld} \nu^{-1/3}\ld^{1/3} \n{P^{Ue}_\nu f_\ld}_3\\
&\ll Z(U).
\end{align*}

Finally, we show that $Z(U)$ is bounded on average. This follows from the averaging estimate~\eqref{pavg}, which gives
\begin{align*}
\E^3[Z(U)] & \ll \n{f}_3 + \sum_{1 \leq \nu \leq \ld} (\nu/\ld)^{1/3} \n{P_\ld f} + \sum_{1 \leq \ld \leq \nu} (\ld/\nu)^{1/3} \n{P_\ld}_3\\
&\ll \n{f}_3 + \sum_{\ld \geq 1} \n{P_\ld f}_3\\
& \ll \n{f}_{B^0_{3,1}}.
\end{align*}
\end{proof}

To state the next lemma, we introduce the mixed-norm notation
\[\n{f}_{L^{p_1}_{x_1} L^{p_2}_{x_2} L^{p_3}_{x_3}}=\ns{\ns{\ns{f(x_1,x_2,x_3)}_{L^{p_3}_{x_3}}}_{L^{p_2}_{x_2}}}_{L^{p_1}_{x_1}}.\]
Given an orthonormal frame $\{e_1,e_2,e_3\}$, we will also use the notation
\[\n{f}_{L^{p_1}_{e_1} L^{p_2}_{e_2} L^{p_3}_{e_3}} = \n{f(x)}_{L^{p_1}_{y_1} L^{p_2}_{y_2} L^{p_3}_{y_3}},\]
where the $L^p$ norms on the right hand side are taken with respect to the coordinates $y_i = x \cd e_i$. Sometimes we will write $L^p_e$ for $L^p_{e_1} L^p_{e_2}$, where $e = e_1 + i e_2$. We will also omit to specify all of the directions $e_i$ when they can be inferred from context.

The next lemma is due to~\cite{Falconer1980}. We will use an easy consequence: if $q$ lies in the range $(3,3/(1-s))$, then
\[\E^{q}[\n{A}_{L^\infty L^2_{Ue}}] \ll \n{A}_{q}\ll \n{A}_{W^{s,3}}.\]
\begin{lemma}
\label{lil2}
Let $f \in L^{p}(\R^3)$, where $p> 3/2$. Assume $f$ is supported in a ball $B(0,1)$. Then
\begin{align*}
\E^{p}[\n{f}_{L^\infty L^1_{Ue}}] & \ll \n{f}_{p}.
\end{align*}
\end{lemma}
\begin{proof}
Let $g = \a{f}$, and let $\eta$ be a mollifier as defined above. Since $g$ is nonnegative, we have
\begin{align*}
\n{g}_{L^\infty_{x_3} L^1_{x_1,x_2}} & = \esssup_{x_3} \int g(x_1,x_2,x_3) \,dx_1\,dx_2\\
&= \esssup_{x_3}\int \int \eta((y_1-x_1) + i(y_2-x_2))\,dy_1\,dy_2 g(x_1,x_2,x_3) \,dx_1\,dx_2\\ 
&=\esssup_{x_3} \int \eta* g(y_1,y_2,x_3) \,dy_1\,dy_2
\end{align*}
Thus we can replace $g$ by $\ti P g$, since
\[\n{g}_{L^\infty_{x_3} L^1_{x_1,x_2}} \leq \n{\ti P g}_{L^\infty_{x_3} L^1_{x_1,x_2}}.\]
More generally, for $U \in O(3)$, we have
\[\n{g}_{L^\infty L^1_{Ue}} \leq \n{\ti P^{Ue} g}_{L^\infty L^1_{Ue}}.\]
Now, since $\ti P^{Ue} g$ is supported in the ball $B(0,2)$, H\"older's inequality implies that
\[\n{\ti P^{Ue} g}_{L^\infty L^1_{Ue}} \ll \n{\ti P^{Ue}g}_{L^\infty L^{p}_{Ue}}.\]
Decompose $g$ as $g = \sum_{\nu,\ld\geq 1} P^{Ue}_\nu P_\ld g$. By abuse of notation, we redefine $P_1$ and $P^{Ue}_1$ as $P_1 = P_{\leq 1}$ and $P^{Ue}_1 = P^{Ue}_{\leq 1}$.  For each of these pieces, we apply Bernstein's inequality in the $U e_3$ direction, which gives
\begin{align*}
\n{\ti P^{Ue} P^{Ue}_{\nu} P_\ld g}_{L^\infty L^{p}_{Ue}}& \ll \ld^{1/p}\n{\ti P^{Ue}P_\nu^{Ue} P_\ld g}_{p}
\end{align*}
By the almost orthogonality bound~\eqref{hfhm}, this implies that
\[\sum_{\nu,\ld\geq 1} \n{\ti P^{Ue} P^{Ue}_{\nu} P_\ld g}_{L^\infty L^{p}_{Ue}}\ll \sum_{\nu,\ld \geq 1} \ld^{1/p}\nu^{-N} \n{P^{Ue}_{\nu} P_\ld g}_{p} \]
Averaging over $O(3)$ using the averaging bound~\eqref{pavg}, we obtain
\begin{align*}
\sum_{\nu,\ld\geq 1}\ld^{1/p} \nu^{-N} \E^3[\n{\ti P^{Ue} P^{Ue}_{\nu} P_\ld g}_{L^\infty L^{p}_{Ue}}]&\ll \sum_{\nu,\ld \geq 1}\ld^{1/p} \j{\ld/\nu}^{2/p-2} \nu^{-N} \n{P_\ld g}_{p} \\
&\ll \sum_{\ld\geq 1} (\ld^{1/p-N} + \ld^{3/p-2}) \n{g}_{p}\\
&\ll \n{g}_{p}.
\end{align*}
\end{proof}

We are now ready to prove estimates in the space $X^{-1/2}_\z$ for some dangerous terms which will appear on the right hand side of the equation $L_{A,q,\z} \psi = \dotsb$ for the remainder $\psi$. The next lemma is fairly straightforward to prove if $a = 1$; in that case it follows from the fact that the $X^{-1/2}_\z$ norm is controlled, on average, by the $H^{-1}$ norm. However, when $a$ is nontrivial, we will have to work harder.
\begin{lemma}
\label{aq}
Let $q \in W^{-1,3}(\R^3)\cap W^{-1,2} (\R^3)$, and suppose that for each $(\tau, U)$ in $\R_+ \times O(3)$ we are given a function $a_{\tau,U}$, such that 
\[M=\sup_{\tau, U}{ (\n{a}_{\infty} + \tau^{-1} \n{\na a}_{\infty} + \tau^{-2} \n{\na^2 a}_\infty +  \n{\na a}_2)} < \infty.\]
Then
\begin{align*}
\label{vda}
\ti\E_K^2\n{a\cd q}_{X_{\tau U (e_1 + ie_2)}^{-1/2}} &\ll M (\log K)^{-1/3} (\n{q}_{W^{-1,2}} + \n{q}_{W^{-1,3}}) + \n{q_{>2^K}}_{W^{-1,2}}.
\end{align*}
\end{lemma}
\begin{proof}
Let $\z = \tau U(e_1+ie_2)$. In what follows we will use $\n{\cd}$ to denote the $X^{-1/2}_\z$ norm. Since we are working with homogeneous norms, it is convenient to redefine all of our dyadic projections by $P_1 = P_{\leq 1}$, $Q_1 = Q_{\leq 1}$ and so on.

At high modulation, we use the high-modulation estimate~\eqref{hmestimate}
\begin{align*}
\n{Q_h (a q)} & \ll \n{a q}_{H^{-1}_\tau}.
\end{align*}
We estimate this using the definition of $H^{-1}_\tau$. For a test function $u$, we have
\begin{align*}
\a{(a q,u)} & = \a{(q,\bar a u)} \\
&\ll \sum_{1 \leq \ld \leq \tau} \a{(q_\ld,\bar  a u)} + \a{(q_{> \tau}, \bar a u)}\\
& \ll \sum_{1 \leq \ld\leq \tau} (\ld/\tau)\n{q_\ld}_{W^{-1,2}} \n{a}_\infty \tau \n{u}_2 + \n{q_{>\tau}}_{W^{-1,2}} \n{au}_{H^1} \\
&\ll M(\sum_{1 \leq \ld\leq \tau}( \ld/\tau)\n{q_\ld}_{W^{-1,2}} + \n{q_{>\tau}}_{W^{-1,2}})\n{u}_{H^1_\tau}
\end{align*}
Thus by duality, we have
\begin{align*}
\n{aq}_{H^{-1}_\tau}& \ll M\sum_{1 \leq \ld\leq \tau} (\ld/\tau)\n{q_\ld}_{W^{-1,2}} + M\n{q_{>\tau}}_{W^{-1,2}}.
\end{align*}
Applying the frequency convolution estimate~\eqref{loggain}, we have
\begin{align*}
\ti\E^2_K[\n{aq}_{H^{-1}_\tau}] & \ll M (\log K)^{-1/2} \n{q}_{W^{-1,2}} + M \n{q_{>2^K}}_{W^{-1,2}}.
\end{align*}

We decompose the low-modulation part as 
\begin{align*}
Q_l(a \cd q ) & = LL + HH ,
\end{align*}
where the low-low part is given by
\begin{align*}
LL & = Q_l(a_{\ll \tau} q_{\ll \tau})
\end{align*}
and the high-high part is given by
\begin{align*}
HH & = \sum_{\ld \ggg \tau} Q_l(a_{\ld} q_{\sim \ld}).
\end{align*}
We further decompose the low-low part as 
\[LL = I + II + III + IV + V,\]
where
\begin{align*}
I & = \sum_{1\leq\ld \ll \tau} \sum_{\mu \geq A(\ld,\tau)} Q_\mu (a_{\ll \tau}  q_{\ld}) \\
II & = \sum_{1\leq\ld \ll \tau} \sum_{\mu < A(\ld,\tau)} Q_\mu (a_{\leq \mu} q_{\ld}) \\
III & = \sum_{1\leq\ld \ll \tau} \sum_{\mu < A(\ld,\tau)} Q_\mu(a_{\geq \ld} q_\ld) \\
IV & = \sum_{1\leq\ld \ll \tau} \sum_{\mu < A(\ld,\tau)} Q_\mu(a_{(\mu,B(\mu,\ld,\tau))} q_\ld) \\
V & = \sum_{1\leq\ld \ll \tau}  \sum_{\mu < A(\ld,\tau)} Q_\mu(a_{[B(\mu,\ld,\tau),\ld)} q_\ld) .
\end{align*}
The constant $c$ is chosen so small that $Q_{\leq \tau/8} (a_\nu q_\ld) = 0$ when $\ld \ggg \tau$ and $\nu < c\ld$. The cutoffs $A(\ld,\tau)$ and $B(\mu,\ld,\tau)$ will be chosen later.

We estimate $\n{I}_{X^{-1/2}_\z}$ by the $L^2$ estimate~\eqref{l2estimate}.
\begin{align*}
\n{I} & \ll \sum_{\mu \geq A(\ld,\tau)} (\mu \tau)^{-1/2} \ld \n{a}_\infty \n{q_{\ld}}_{W^{-1,2}}\\
& \ll \n{a}_\infty \sum_{\ld\ll \tau} (\ld^2/\tau)^{1/2} A(\ld,\tau)^{-1/2} \n{q_\ld}_{W^{-1,2}}.
\end{align*}
Taking $A(\ld,\tau) = \ld^{2-2\e} \tau^{-1+2\e}$, we have
\begin{align*}
\n{I}&\ll \n{a}_\infty \sum_{\ld \ll \tau} (\ld/\tau)^{\e} \n{q_\ld}_2\\
\ti\E^2_K[\n{I}] & \ll (\log K)^{-1/2} \n{a}_\infty \n{q}_{W^{-1,2}}.
\end{align*}
For $\n{II}$ we note that multiplication by $a_{\ll \mu}$ shifts the Fourier support by at most $\mu$. Thus we have $Q_\mu(a_{\leq \mu}q_{>\mu}) = Q_\mu(a_{\leq \mu} Q_{\ll \mu}q_{>\mu})$. By the $L^2$ estimate~\eqref{l2estimate} and the averaging estimate~\eqref{qavg}, we have
\begin{align*}
\E^2_{\tau_*}[\n{II}] &\ll  \sum_{\mu \ll A(\ld,\tau_*)}(\mu \tau_*)^{-1/2} \n{a}_{\infty} \E^2_{\tau_*}[\n{ Q_{\ll \mu} q_{\ld}}_{2}] \\
&\ll M\sum_{\ld \ll \tau_*} \sum_{\mu \ll A(\ld,\tau_*)} (\mu/\tau_*)^{1/2} \n{q_\ld}_{W^{-1,2}}\\
&\ll M \sum_{\ld \ll \tau_*} (\ld/\tau_*)^{1-\e} \n{q_\ld}_{W^{-1,2}}\\
\ti\E^2_K[\n{II}] &\ll M (\log K)^{-1/2} \n{q}_{W^{-1,2}}.
\end{align*}
For $\n{III}$ we use the Strichartz estimate~\eqref{dualbandstrichartz}:
\begin{align*}
\n{III} & \ll \sum_{\ld \ll \tau}\sum_{\mu < A(\ld,\tau)} \sum_{\nu \geq \ld}(\mu/\tau)^{1/3}\ns{ a_{\nu} q_\ld}_{6/5} \\
&\ll \sum_{\ld \ll \tau}\sum_{\mu<A(\ld,\tau)}\sum_{\nu \geq \ld}(\mu/\tau)^{1/3} (\ld/\nu)  \n{\na a_\nu}_2 \n{q_\ld}_{W^{-1,3}}\\
&\ll \sum_{\ld \ll \tau}  (\ld/\tau)^{(2-2\e)/3} \n{\na a}_2 \n{q_\ld}_{W^{-1,3}} \\
\ti\E^3_K[\n{III}]&\ll M (\log K)^{-1/3} \n{q}_{W^{-1,3}}.
\end{align*}
For the terms in $\n{IV_\mu}$ we can use the identity $Q_\mu(a_{(\mu,B)} q_\ld) = Q_\mu(a_{(\mu,B)} Q_{\ll B} q_\ld)$. Thus by the $L^2$ estimate~\eqref{l2estimate} and the averaging estimate~\eqref{qavg} we have, with $B = \min{\{\ld, \mu^{1/2}\tau^{1/2-2\e} \ld^{2\e}\}}$,
\begin{align*}
\n{Q_\mu(a_{(\mu,B)}  q_\ld)} & \ll (\mu \tau)^{-1/2}\n{a}_\infty \n{Q_{\ll B} q_\ld}_2\\
\E^2_{\tau_*}[\n{Q_\mu(a_{(\mu,B)}  q_\ld)}] & \ll M(\ld/\tau_*)^{2\e} \n{q_\ld}_{W^{-1,2}}
\end{align*}
Summing over $\mu$, we obtain
\begin{align*}
\E^2_{\tau_*}[\n{IV}] & \ll M(\log \tau_*)^{1/2} \sum_{\ld \ll \tau_*}(\ld/\tau_*)^{2\e} \n{q_\ld}_{W^{-1,2}}\\
\ti\E^2_K[\n{IV}]&\ll M (\log K)^{-1/2}\n{q}_{W^{-1,2}}.
\end{align*}

For $\n{V_\mu}$, we use the identity $Q_\mu(a_\nu q_\ld) = Q_\mu(a_\nu Q_{\ll \nu} q_\ld)$. Using the Strichartz estimate~\eqref{dualbandstrichartz} and the averaging estimate~\eqref{qavg}, we obtain
\begin{align*}
\n{Q_\mu(a_\nu  q_\ld)} &\ll (\mu/\tau)^{1/3} \n{a_\nu Q_{\ll \nu} q_\ld}_{6/5} \\
&\ll (\mu/\tau)^{1/3} \n{a_\nu}_2 \n{Q_{\ll \nu} q_\ld}_3\\
&\ll (\mu/\tau)^{1/3} \nu^{-1} \n{\na a_\nu}_2 \n{Q_{\ll \nu} q_\ld}_3\\
\E^3_{\tau_*}[\n{Q_\mu(a_\nu  q_\ld)}] & \ll (\mu/\tau_*)^{1/3} (\ld/\nu)^{1/2} (\tau_*/\nu)^{2\e} M \n{q_\ld}_{W^{-1,3}}
\end{align*}
Summing over $\nu$, we have
\begin{align*}
\E^3_{\tau_*}[\n{Q_\mu(a_{[B,\ld)}  q_\ld)}] & \ll (\mu/\tau_*)^{1/3} (\ld/B)^{1/2} (\tau_*/B)^{2\e}  M \n{q_\ld}_{W^{-1,3}}\\
&\ll \mu^{1/12-\e} \tau_*^{-7/12 + 2\e + 4\e^2} \ld^{1/2 - \e - 4 \e^2}  M \n{q_\ld}_{W^{-1,3}}\\
&\ll (\mu/\ld)^{1/12-\e} (\ld/\tau_*)^{7/12 - 2 \e - 4\e^2} M \n{q_\ld}_{W^{-1,3}}
\end{align*}
Summing over $\mu$ and applying the frequency convolution estimate~\eqref{loggain}, this gives
\begin{align*}
\E^3_{\tau_*}[\n{V}] & \ll \sum_{\ld \ll \tau_*} (\ld/\tau_*)^{\al}  M \n{q_\ld}_{W^{-1,3}} \\
\ti \E^3_K[\n{V}]& \ll  M(\log K)^{-1/3} \n{q}_{W^{-1,3}}.
\end{align*}

Finally, we estimate the high-high terms. When the modulation is sufficiently small, we use the Strichartz estimate~\eqref{dualbandstrichartz}
\begin{equation}
\label{smallmod}
\begin{aligned}
\n{Q_{\leq C} (a_\ld q_{\sim \ld})} &\ll  (C /\tau)^{1/3} \ld \n{a_\ld}_2 \n{q_{\sim \ld}}_{W^{-1,3}}\\
&\ll (C/\tau)^{1/3} M \n{q}_{W^{-1,3}}
\end{aligned}
\end{equation}
When the modulation is large, we use the $L^2$ estimate~\eqref{l2estimate} and then estimate $\n{a}_6$ by interpolation.
\begin{equation}
\label{largemod}
\begin{aligned}
\n{Q_{> C} (a_\ld q_{\sim \ld})} & \ll (C \tau)^{-1/2}\ld  \n{a_\ld}_6 \n{q_{\sim \ld}}_{W^{-1,3}}\\
&\ll (C \tau)^{-1/2} \ld \n{a_\ld}_\infty^{2/3} \n{a_\ld}_2^{1/3} \n{q_{\sim \ld}}_{W^{-1,3}}\\
&\ll (C \tau)^{-1/2} \ld^{2/3} (\tau/\ld)^{4/3}  M\n{q}_{W^{-1,3}}
\end{aligned}
\end{equation}
Here we use that $\n{\na^2 a}_\infty \ll \tau^2 M$. Let $C = \tau \ld^{-\e}$. Summing the inequalities~\eqref{smallmod} and~\eqref{largemod} over $\ld \ggg \tau$, we obtain
\begin{align*}
\n{HH} & \ll  \sum_{\ld \ggg \tau} (\ld^{-\e/3} + \tau^{1/3} \ld^{-2/3 + \e})M\n{q}_{W^{-1,3}}\\
\ti\E_K[\n{HH}]&\ll 2^{-\e K/3} M \n{q}_{W^{-1,3}}.
\end{align*}

\end{proof}
In the next lemma, we make use of the relationship between the operator $\pb_e$ and the operator $\lp_\z$.
\begin{lemma}
\label{lpchia}
Fix $s>1$. Let $B = B(0,1)$. Let $A$ be a smooth function supported in $\ff 1 2 B$, and let $\chi$ be a cutoff supported in $B$ such that $\chi = 1$ on $\f1 2 B$. Let $a = \exp(\pb_e^{-1} A)$. Then
\begin{align*}
\n{\lp (\chi a)}_{X^{-1/2}_\z} & \ll_{s} (1+\n{\na \pb_e^{-1} A}_{L^2(B)} + e^{\n{\pb_e^{-1} A}_\infty} + \n{A}_{L^\infty L^2_e} \\
&\qq+ \n{\j{\na_1}^{-1/2+s} \j{\na_2}^{-1/2+s}\na A}_2 + \n{\j{\na_{1,2}}^{-1+s} \na A}_{2} )^4
\end{align*}
\end{lemma}
\begin{proof}
As in the previous lemma, we redefine $P_1$ as $P_{\leq 1}$ and so on. 

At high modulation we use the high-modulation estimate~\eqref{hmestimate}.
\begin{align*}
\n{Q_h \lp(\chi a)}_{X^{-1/2}} & \ll \n{\chi a}_{H^1}\\
&\ll \n{a}_{H^1(B)}.
\end{align*}
It remains to consider the low-modulation part of $\chi a$. By the $L^2$ estimate~\eqref{l2estimate},
\[\n{Q_l \lp (\chi a)}_{X^{-1/2}}^2 \ll \sum_{1 \leq \mu \ll \tau} \sum_{ \ld \ll \tau} (\mu \tau)^{-1} \n{Q_\mu P^e_\ld \lp (\chi a)}_2^2.\]
Now we observe that
\begin{align*}
\a{\xi}^2 &= 2 i \z \cd \xi + p_\z(\xi) \\
&\ll \tau \a{\xi \cd e} + \tau d(\xi,\Si).
\end{align*}
Thus, when $\ld \leq \mu$, the symbol of $Q_\mu P_\ld^e \na$ is bounded by $(\mu\tau)^{1/2}$. It follows that 
\begin{align*}
\sum_{1 \leq \mu \ll \tau} \sum_{\ld \leq \mu} (\mu \tau)^{-1} \n{Q_\mu P^e_\ld \lp (\chi a)}_2^2 & \ll \sum_{1 \leq \mu \ll \tau} \sum_{\ld \leq \mu} \n{Q_\mu P_\ld^e \na(\chi a)}_2^2 \\
&\ll \n{\na(\chi a)}_2^2\\
&\ll \n{a}_{H^1(B)}^2.
\end{align*}
It remains to control the terms where $\ld > \mu$. In this case the symbol of $Q_\mu P_\ld^e\na$ is bounded by $(\ld\tau)^{1/2}$, so we have
\begin{align}
\label{lgtmu}
\sum_{1 \leq \mu < \ld \ll \tau}(\mu \tau)^{-1/2} \n{Q_\mu P_\ld^e \lp (\chi a)}_2 & \ll \sum_{1 \leq \mu \leq \tau/8} \sum_{\mu < \ld}\mu^{-1/2} \ld^{-1/2}\n{ Q_\mu P^e_\ld \na \pb_e (\chi a)}_2
\end{align}
Since the commutator $[\na \pb_e,\chi]$ satisfies the bound
\[\n{[\na\pb_e,\chi]a}_2 \ll \n{a}_{H^1(B)},\]
we may replace $\na \pb_e(\chi a)$ with $\chi \na\pb_e a$ on the right-hand side of~\eqref{lgtmu}. Now we use the definition of $a$ to write
\[\na \pb_e a = \na (A a) = \na A a + A \na a.\]
For $A \na a$, we apply Bernstein's inquality in the $e_1$ and $e_2$ directions and use the identity $\na a = \pb_e^{-1} A \cd a$.
\begin{align*}
\mu^{-1/2} \ld^{-1/2} \n{Q_\mu P^e_\ld (\chi A \na a)}_2 & \ll (\mu \ld)^{-s/2} \n{ Q_\mu P^e_\ld(\chi\cd A\na a)}_{L^{2} L^{1/(1-s/2)}_{e}} \\
&\ll (\mu\ld)^{-s/2} \n{A}_{L^\infty L^2_{e}}\n{\chi\na a}_{L^2 L^{2/(1-s)}_e}\\
&\ll (\mu \ld)^{-s/2} \n{A}_{L^\infty L^2_e}\n{a}_\infty \n{\chi\na \pb_e^{-1} A}_{L^2 L^{2/(1-s)}_e}
\end{align*}
Since $s>0$, we can sum the right-hand side over $\mu$ and $\ld$ as long as the last factor is bounded. To check this, we use the localization estimate~\eqref{dbploc} and Sobolev embedding.
\begin{align*}
\n{\chi\na \pb_e^{-1} A}_{L^2 L^{2/(1-s)}_e} & \ll \n{\na P_{\leq 1}^{e} A}_{L^2 L^{2/(1-s)}_e} + \n{\na\pb^{-1}_e P_{>1}^e A}_{L^2 L^{2/(1-s)}_e} \\
&\ll \n{\na P^e_{\leq 1} A}_2 + \n{\na \j {\na_{1,2}}^s \pb_e^{-1} P^e_{>1} A}_{2} \\
&\ll \n{\j{\na_{1,2}}^{s-1} \na A}_2.
\end{align*}

For $(\chi a)\na A $, use the LP dichotomy formula:
\[P^e_{\ld} ((\chi a )\na A) = \sum_{\k \ll \ld}P^e_\ld( P^e_\k(\chi a)\cd P^e_{\ll \ld} \na A) + P^e_{\ld}(\sum_{\eta \gg \ld} P^e_\eta (\chi a) \cd P^e_{\sim \eta} \na A).\]
For the low-low terms, we have two cases. When $\k \leq \mu$, we use the identity $Q_\mu (P_{\leq \mu}^e f\cd g) = Q_\mu (P_{\leq \mu}^e f \cd  P^{e_1}_{\ll \mu} g)$. Thus we have
\begin{align*}
\mu^{-1/2} \ld^{-1/2} \n{Q_\mu P^e_\ld (P_{\leq \mu}^e (\chi a)\cd P^e_{\ll \ld} P^{e_1}_{\ll \mu} \na A)}_2 & \ll \n{a}_\infty (\mu \ld)^{-1/2} \n{P^{e_2}_{\ll \ld} P^{e_1}_{\ll \mu} \na A}_2
\end{align*}
Summing over $\mu$ and $\ld$, we obtain
\begin{align*}
\sum_{1\leq \mu,\ld \leq \tau/8}\dotsb & \ll  \n{a}_\infty \n{\j{\na_{1}}^{-1/2+s} \j{\na_{2}}^{-1/2+s} \na A}_2.
\end{align*}
When $\k > \mu$, we have instead $Q_\mu(P_\k^e f \cd g) = Q_\mu(P_\k^e f \cd P^{e_1}_{\ll \k} g)$. Then
\begin{align*}
\mu^{-1/2} \ld^{-1/2} \n{Q_\mu(P_{\k}^e (\chi a) P^e_{\ll \ld} P^{e_1}_{\ll \k} \na A)}_2 &\ll \ld^{-1/2} \n{P_\k^e (\chi a) P^e_{\ll \ld} P^{e_1}_{\ll \k} \na A)}_{L^2_{e_3,e_2}L^1_{e_1}} \\
&\ll \ld^{-1/2} \n{P_\k^e (\chi a)}_{L^\infty_{e_3} L^\infty_{e_2} L^2_{e_1}} \n{P^{e_2}_{\ll \ld} P^{e_1}_{\ll \k} \na A}_{L^2}\\
&\ll \ld^{-1/2}\k^{-1/2} \n{P_\k^e \pb_e (\chi a)}_{L^\infty L^2_{e}} \n{P^{e_2}_{\ll \ld} P^{e_1}_{\ll \k} \na A}_{L^2}.
\end{align*}
Summing over $\k$, $\mu$, and $\ld$, we obtain 
\[\sum_{1 \leq \mu< \k\leq \ld \leq \tau/8} \dotsb \ll  \n{a}_\infty(1 + \n{A}_{L^\infty L^2_{e}}) \n{\j{\na_1}^{-1/2+s} \j{\na_2}^{-1/2+s} \na A}_{L^2}.\]
For the high-high terms, we use Bernstein's inequality and then transfer the $\pb^{-1}_e$ from $a$ to $A$:
\begin{align*}
\mu^{-1/2} \ld^{-1/2} \n{Q_\mu P_\ld^e(P_\eta^e (\chi a) \cd P^e_{\sim \eta} \na A)}_{2} & \ll  \n{P_\eta^e(\chi a) \cd P^e_{\sim \eta} \na A}_{L^2 L^{1}_{e}}\\
&\ll \n{P_\eta^e(\chi a)}_{L^\infty L^2_e} \n{P^e_{\sim \eta} \na A}_{2} \\
&\ll \eta^{-1}\n{P_\eta^e\pb (\chi a)}_{L^\infty L^2_e} \n{P^e_{\sim \eta} \na A}_{2}\\
&\ll \eta^{-s}\n{a}_\infty (1+\n{A}_{L^\infty L^2_{e}}) \n{\j{\na_{1,2}}^{-1+s} \na A}_{L^2}.
\end{align*}
The sum of the right hand side over $\eta \geq \ld \geq \mu\geq 1$ is bounded, and the proof is complete.
\end{proof}

\section{Solvability of $L_{A,q,\z}$}
\label{solvabilityoflaqz}
Now we show that on average, the terms in $L_{A,q,\z} + \lp_\z$ are all perturbative.
\begin{lemma}
\label{solvability}
Let $e_1$ be a fixed unit vector in $\R^3$, and let $\eta$ be a vector in $\R^3$ such that $\a{\eta}\leq 1$. Define the operator norm $\tn{\cd}_{\tau,U }$ by $\tn{T} = \n{\cd}_{X^{1/2}_{\z(\tau,U)} \to X_{\z(\tau,U)}^{-1/2}}$, where $\z(\tau,U) = \tau U(e_1 + i \eta)$. Suppose $A \in L^3(\R^3)$. For every dyadic integer $\ld$ such that $1\leq \ld \leq 100\tau$ we have
\begin{equation}
\label{agradestimate}
\E^3[\tau \tn{A_\ld}_{\tau,U }+ \tn{A_\ld\cd\na}_{\tau,U}] \ll \min\{\j{\log_+ \ld}, \j {\log_+ \tau/\ld}\}^{1/3} \n{A_\ld}_{L^3}
\end{equation}
On the other hand, we have the high-frequency estimate
\begin{equation}
\label{ahf}
\tn{\na A_{> 100\tau}} + \tn{A_{> 100\tau} \cd \na} \ll \n{A_{>100 \tau}}_{L^3}.
\end{equation}
Finally, for $q \in W^{-1,3}$ we have
\begin{equation}
\label{potentialestimate}
\ti\E^3_K[\tn{q}_{\tau,U}] = (\log K)^{-1/3} \n{q}_{W^{-1,3}} + \n{q_{> 100\tau}}_{W^{-1,3}}.
\end{equation}
\end{lemma}
\begin{proof}
It is convenient to use a bilinear characterization of the $\tn{\cd}$ norm
\[\tn{A} = \sup\{\a{\j{A  u,v}}: \n{u}_{X^{1/2}_\z}=\n{v}_{X^{1/2}_\z}=1\}.\]
Note that 
\[\j{A u, v} = \j{A e^{-i v \cd x} u,e^{-i v \cd x} v},\]
and that the $X^{1/2}_\z$ spaces have the modulation invariance
\[\n{e^{-i v \cd x} u}_{X^{1/2}_\z} \sim \n{u}_{X^{1/2}_{\z + i v}}.\]
Thus we may as well assume that $\eta$ is zero.

Decompose $u$ and $v$ into low and high modulation parts:
\begin{align*}
\j{A u, v} & = \j{A Q_h u, v} + \j{A Q_l u, Q_h v} + \j{A Q_l u, Q_l v}.
\end{align*}
The terms with $Q_h$ can be estimated by the high-modulation estimate~\eqref{hmestimate} and the Strichartz estimate~\eqref{strichartz}. For example,
\begin{align*}
\tau \a{\j{A Q_h u,v}} & \ll\tau \n{A}_3 \n{Q_h u}_2\n{v}_6 \\
&\ll \n{A}_3 \n{Q_h u}_{X^{1/2}_\z} \n{v}_{X_\z^{1/2}} 
\end{align*}

It remains to estimate the low modulation terms. Write
\begin{align}
\label{qsum}
\j{A_\ld \cd Q_l u, Q_l v} & = \sum_{\tau/8\geq\mu,\nu\geq 1}\int A_\ld \cd Q_\mu u \cd \bar{Q_\nu v}\,dx ,
\end{align}
where we make the notational convention that $Q_1 = Q_{\leq 1}$. 
Note that when $\ld\geq 100\tau$ these terms are all zero, so for the high-frequency estimate~\eqref{ahf} there is nothing left to prove.

Set
\begin{align*}
a_\mu & = \n{Q_\mu u}_{X^{-1/2}_\z}\\
b_\mu & = \n{Q_\mu v}_{X^{-1/2}_\z},
\end{align*}
and
\[B_{\mu, \nu} = \tau \as{\int  A_\ld \cd Q_\mu u \cd \bar{Q_\nu v}\,dx}.\]
We claim that 
\begin{equation}
\label{abimpliesb}
\sum_\mu (a_\mu^2 +  b_\mu^2)  \ll 1 \text{ implies that } \sum_{\mu, \nu} B_{\mu, \nu} \ll Z(U),
\end{equation}
where
\[\E^3[Z(U)] \ll \n{A_\ld}_{L^3}.\]
By symmetry, it suffices to treat the terms where $\mu \leq \nu$. Since $Q_\mu u \cd \bar{Q_\nu v}$ has Fourier support in the set $\{\a{\xi_1} \leq 2\nu\}$, we have
\[\int  A_\ld \cd Q_\mu u \cd \bar{Q_\nu v}\,dx = \int P^{Ue_1}_{\leq 8\nu}  A_\ld \cd Q_\mu u \cd \bar{Q_\nu v}\,dx.\]
Suppose first that $\ld^2 > \mu\tau$. In this case we use H\"older's inequality and estimate $Q_\mu u$ by using the Strichartz estimate~\eqref{bandstrichartz}:
\begin{align*}
B_{\mu,\nu} & \ll\tau \n{P^{Ue_1}_{\leq 8\nu} A_\ld}_3 \n{Q_\mu u}_6 \n{Q_\nu v}_2\\
&\ll \n{P^{Ue_1}_{\leq 8\nu} A_\ld}_3 \tau (\mu/\tau)^{1/3} (\nu \tau)^{-1/2} a_\mu b_\nu.
\end{align*}
By Young's inequality, we have
\[\sum_{\mu\leq \nu} (\mu/\nu)^{1/12} a_\mu b_\nu \ll 1,\]
so that the sum over $\ld^2 > \mu\tau$ is bounded by
\begin{align*}
\sum_{\substack{\mu \leq \nu\\\ld^2 > \mu\tau}} B_{\mu,\nu} &\ll \sup_{\substack{\nu \geq \mu\\ \ld^2 > \mu \tau}} (\ld/\nu)^{1/3} (\mu/\nu)^{1/12}(\mu \tau/\ld^2)^{1/6} \n{P^{Ue_1}_{\leq 8\nu} A_\ld}_3\\
&\ll  Z_1(U) + Z_2 (U)+ \n{A_\ld}_3,
\end{align*}
where $Z_1(U)$ and $Z_2(U)$ are given by
\begin{align*}
\label{f1}
Z_1(U)^3 &= \sum_{\max\{1, \ld^2/\tau\} \leq \nu \leq \ld} (\ld/\nu)(\ld^2/\nu\tau)^{1/12}  \n{P^{Ue_1}_{\leq 8\nu} A_\ld}^3_3\\
\notag
Z_2(U)^3 & = \sum_{\nu \leq \ld^2/\tau} (\ld/\nu) (\nu \tau/\ld^2)^{1/2} \n{P^{U e_1}_{\leq 8\nu} A_\ld}_3^3.
\end{align*}
Now we check that $Z_1$ and $Z_2$ are bounded on average by applying the averaging estimate~\eqref{pavg}.
\begin{align*}
\E^3[Z_1(U)]^3 & \ll  \sum_{\max\{1, \ld^2/\tau\} \leq \nu \leq \ld}(\ld^2/\nu\tau)^{1/12} \n{A_\ld}_3^3 \\
& \ll \n{A_\ld}_3^3\\
\E^3[Z_2(U)]^3 & \ll \sum_{\nu \leq \ld^2/\tau} (\nu \tau/\ld)^{1/2}/\ld\n{A_\ld}_3^3 \\
& \ll \n{A_\ld}_3^3
\end{align*}

Next, we treat the case $\ld \leq (\mu\tau)^{1/2}$. We subdivide the set 
\[E_l = \{\xi: d(\xi,\Si_\z) \leq \tau/8\}\]
into $M=\floor{(\tau/\mu)^{1/2}}$ sectors $S_k$, defined for $k=0,\dotsc,M-1$ by
\[S_k = E_l \cap \{(\xi_1, r \cos \th, r \sin \th): \th \in (2\pi/M)[k, k+1), r \in \R_+\}.\]
Let $R_k$ be Fourier projection onto $S_k$. The distance between two points in $E_l$ is bounded below by $\tau \th$, where $\th$ is the angular separation between the points. Thus for any two sectors $S_j$ and $S_k$, we have
\[d(S_j,S_k) \gg (\mu \tau)^{1/2} d_M(j,k),\]
where $d_M(j,k) = \min \{\a{j-k},M-\a{j-k}\}$. Since $A_\ld \cd R_k f$ has Fourier support in the set $\{S_k + B(0, 2\ld)\}$, we find that the inner product $\j{A_\ld \cd R_k f, R_j g}$ vanishes unless $\a{d_M(j,k)} \leq C$, so that
\[B_{\mu,\nu} \ll \tau \sum_{d_M(j,k) \leq C} \a{\j{P^{Ue_1}_{\leq 8\nu}A_\ld \cd R_k Q_\mu u, R_j Q_\nu v}}.\]
The Fourier support of $R_k Q_\mu u$ is contained in a rectangle of size proportional to $\mu^{1/2}\tau^{1/2}\times \mu \times \mu$. Thus, applying H\"older and Bernstein in each direction separately, we obtain
\begin{align*}
\n{P^1_{\leq 8\nu}A_\ld\cd R_k Q_\mu u}_{L^2} & \ll \n{P^{Ue_1}_{\leq 8\nu}A_\ld}_{L^\infty L^3 L^3} \n{R_k Q_\mu u}_{L^2 L^6 L^6}\\
&\ll \ld^{1/3} \mu^{2/3} \n{P^{Ue_1}_{\leq 8\nu}A_\ld}_{L^3} \n{R_k Q_\mu u}_{L^2}.
\end{align*}
Now apply Cauchy-Schwarz to the sum over $j$ and $k$. This gives
\begin{align*}
B_{\mu,\nu}& \ll \ld^{1/3} \mu^{2/3} \n{P^{Ue_1}_{\leq 8\nu} A_\ld}_3 \sps{\sum_k \n{R_k Q_\mu u}_2^2}^{1/2} \sps{\sum_k \n{R_k Q_\nu v}^2_2}^{1/2}\\
&\ll \tau \mu^{2/3} \ld^{1/3}\n{P^{Ue_1}_{\leq 8\nu} A_\ld}_{L^3} \n{Q_\mu u}_{L^2}\n{Q_\nu v}_{L^2}\\
& \ll (\mu/\nu)^{1/6} (\ld/\nu)^{1/3}a_\mu b_\nu \n{ P^{Ue_1}_{\leq 8\nu} A_\ld}_{L^3} 
\end{align*}
Thus
\begin{align*}
\sum_{\substack{\ld^2 \leq \mu \tau\\\mu \leq \nu}} B_{\mu,\nu}& \ll \sup_{\nu \geq \max\{\ld^2/\tau, 1\}} {(\ld/\nu)^{1/3} \n{P^{Ue_1}_{\leq 8\nu} A_\ld}_3}\\
&\ll Z_3 + \n{A_\ld}_3,
\end{align*}
where $Z_3$ is given by
\[Z_3(U)^3 = \sum_{\max\{\ld^2/\tau,1\} \leq \nu \leq \ld} (\ld/\nu)^{1/3} \n{P^{Ue_1}_{\leq 8\nu}A_\ld}_3.\]
Applying the averaging estimate~\eqref{pavg}, we have
\begin{align*}
\E^3[Z_3(U)]^3& \ll  \sum_{\max\{\ld^2/\tau,1\} \leq \nu \leq \ld} \n{A_\ld}_3\\
&\ll \min\{\j{\log_+ \ld}, \j{\log_+ \tau/\ld}\}^{1/3} \n{A_\ld}_3.
\end{align*}
This proves the claim~\eqref{abimpliesb}, which shows that 
\[\E^3[\tau \tn{A_\ld}_{\tau,U }] \ll \min\{\j{\log_+ \ld}, \j {\log_+ \tau/\ld}\}^{1/3} \n{A_\ld}_{3}.\]
The estimate for $\tn{A\cd \na}$ now follows, since
\begin{align*}
\n{A \cd \na u}_{X^{-1/2}_\z} & \ll \tn{A} \n{\na Q_l u}_{ X^{1/2}_\z} + \n{A\cd \na Q_h u}_{X^{-1/2}_\z}.
\end{align*}
Since $Q_l$ localizes to frequencies $\a \xi \ll \tau$, we can estimate the first term using the bound 
\[ \n{\na Q_l u}_{X^{1/2}_\z} \ll \tau \n{u}_{X^{1/2}_\z}.\]
For second term we apply the Strichartz estimate~\eqref{dualstrichartz} and the high-modulation estimate~\eqref{hmestimate}
\begin{align*}
\n{A \cd \na Q_h u}_{X^{-1/2}_\z} & \ll \n{A \cd \na Q_h u}_{L^{6/5}} \\
&\ll \n{A}_3 \n{\na Q_h u}_{L^2}\\
&\ll \n{A}_3 \n{u}_{X^{1/2}_\z}.
\end{align*}
Thus we have
\[\E^3[\tn{A_\ld \cd \na}] \ll \E^3[\tau \tn {A_\ld}] + \n{A_\ld}_3 \ll \j{\log_+ \tau/\ld}^{1/3} \n{A_\ld}_3.\]

Finally we derive the estimate~\eqref{potentialestimate} for $\tn{q}$. By Bernstein's inequality, we can control $\tn{q_{\leq 1}}$ by 
\begin{align*}
\tn{q_{< 1}} &\ll \tau^{-1} \n{q_{\leq 1}}_\infty \\
&\ll \tau^{-1} \n{q}_{W^{-1,n}}.
\end{align*}
On the other hand, the estimate~\eqref{agradestimate} for $\tau \tn{A_\ld}$ applied to $q_\ld$ gives
\begin{align*}
\E^3[\tn{q_{[1,100\tau]}}] & \ll \sum_{1 \leq \ld \leq 100 \tau} (\ld/\tau) \j{\log_+ \tau/\ld}^{1/3} \n{q_\ld}_{W^{-1,3}} \\
&\ll \sum_{1 \leq \ld \leq 100\tau} (\ld/\tau)^{1-\e} \n{q_\ld}_{W^{-1,3}}.
\end{align*}
Thus the frequency convolution estimate~\eqref{loggain} gives
\[\ti\E_K^3[\tn{q_{[1,100\tau]}}] \ll (\log K)^{-1/3} \n{q}_{W^{-1,3}}.\]
Together with the high-frequency estimate~\eqref{ahf}, this gives the desired bound~\eqref{potentialestimate} for $\tn{q}$.
\end{proof}

\section{Proof of the main theorem}
\label{proof}
We will need the following Poincar\'e lemma:
\begin{lemma}
\label{poincare}
Suppose that $A \in L^3(\R^3)$ is compactly supported, and that 
\[\curl A = 0.\]
Then there exists $\psi \in W^{1,3}(\R^3)$ supported in $\supp A$ such that 
\[\na \psi = A.\]
\end{lemma}
\begin{proof}
If $\psi$ exists then
\[\lp \psi = \div(\na \psi) = \div(A).\]
Thus we set
\[\psi (x) = c \int \sum_i \pd_i K(y) A_i(x-y)\,dy,\]
where $K(y) \sim \a{y}^{-1}$ is the fundamental solution of the Laplacian. Since $\a{\pd_i K} \ll \a{y}^{-2}$, the kernel is locally integrable. By Young's inequality and the fact that $A$ is compactly supported, we have $\psi \in L^3_{\loc}$. Furthermore, the function $\psi$ is smooth away from $\supp A$ and decays to zero at infinity. The vector Laplacian is given by 
\[\lp = \grad \cp \div - \curl \cp \curl.\]
Since $\curl A$ and $\curl \psi$ are both zero, we have
\begin{align*}
\lp (A - \na \psi) = \na (\div A - \lp \psi)=0.
\end{align*}
Since $A - \na \psi$ vanishes at infinity, this implies, by the maximum principle, that $\na \psi = A$. In particular, $\na \psi = 0$ away from $\supp A$, and since $\psi$ decays at infinity we conclude that $\psi = 0$ away from $\supp A$.
\end{proof}

\thetheorem*
\begin{proof}
Let $B = B(0,1)$. We construct solutions in $H^1(B)$ to the Schr\"odinger equation $L_{A,q} u = 0$ of the form
\begin{equation}
\label{cgosolution}
u  =  e^{x\cd \ti \z}(a + \psi).
\end{equation}
Let $\chi,\ti \chi \in C^\infty_0(B)$ be cutoff functions satisfying $\chi = \ti \chi = 1$ in $\f 1 2 B$ and $\chi \ti \chi = \ti \chi$. We construct $\psi$ by solving the equation
\begin{equation}
L_{A,q,\ti \z} \psi = \ti \chi (F+G),
\end{equation}
where
\begin{equation}
\label{laqzdef}
L_{A,q,\ti z} = -\lp_\z - 2 i \z \cd A + D\cd A + 2 A \cd D + q ,
\end{equation}
\begin{equation}
\label{Fdef}
F =  -\lp (\chi a)  + (D \cd A) a + 2 A \cd Da +  A^2a + qa,
\end{equation}
and
\begin{equation}
\label{Gdef}
G = -2\ti \z \cd \na a - 2 i\ti \z \cd A .
\end{equation}
In order to eliminate the terms of order $\tau$ in $G$, we let $a = e^{-i\phi_\z}$, where
\begin{equation}
\label{phidef}
\phi_\z = \pb_\z^{-1}(\chi \z\cd A_{\leq 100\tau}),
\end{equation}
where $\z \in \C^n$ will be chosen such that $\a{\z - \ti \z} = O(1)$. With this choice of $\phi$, the function $G$ satisfies 
\begin{align*}
\ti \chi G & = \ti \chi( -2i\z\cd A_{>100\tau} a - (\ti \z- \z) \cd (2\na a + 2 i Aa)).
\end{align*}

We choose the parameters $\z_i,\ti \z_i$ as follows: Fix a radius $r\in [1,2]$ and an orthonormal frame $\{e_1,e_2, e_3\}$, and define 
\begin{align*}
\z_1(\tau,U) & = \tau U(e_1 + i e_2)\\
\z_2(\tau, U) & = -\z_1(\tau, U)\\
\ti \z_1 (\tau,U) &:= \tau Ue_1 + i \sps{\ff r 2 Ue_3 + \sqrt{\tau^2 - r^2}Ue_2}\\
\ti \z_2 (\tau,U) &:= -\tau U e_1 + i\sps{\ff r 2 Ue_3 - \sqrt{\tau^2 - r^2}Ue_2}
\end{align*}
Note that $\a{\z_i - \ti \z_i} \ll 1$. In particular, the spaces $X^b_{\z_i}$ and $X^b_{\ti \z_i}$ have equivalent norms.

Let $\ti A_1 = A_1$ and $\ti A_2 = -A_2$, define $F_i,G_i,\phi_{i,\z}$ as in~\eqref{Fdef},~\eqref{Gdef},~\eqref{phidef} by replacing $A$ with $\ti A_i$. 
Let
\begin{align*}
M &= 1 + \sum_i (\n{A_i}_{W^{s,3}} + \n{q_i}_{W^{-1,3}})\\
Z_0 &= 1 + \sum_i \Big(\sum_{k=0,1,2} \tau^{-k} \n{\na^k \phi_{i,\z_i}}_\infty + \n{\na \phi_{i,\z_i}}_{L^2(B)}\\
&\qq +\n{A_i}_{L^\infty L^2_e} + \n{\j{\na_1}^{-1/2+s} \j{\na_2}^{-1/2+s}\na A_i}_2 + \n{\j{\na_{1,2}}^{-1+s} \na A_i}_{2}  \Big)\\
Z_1 &= \sum_{i,l,m} (\tau \n{A_i}_{X^{1/2}_{\z_j}\to X^{-1/2}_{\z_j}}+\n{q_{l,m}}_{X^{1/2}_{\z_j}\to X^{-1/2}_{\z_j}})\\
Z_2 & = \sum_{i,j,l,m} \n{\chi a_i q_{j,l}}_{X^{-1/2}_{\z_m}}  .
\end{align*}
Here the $q_{j,l}$ are all the terms that are bounded in $W^{-1,3}$, namely
\begin{align*}
q_{j,1} & = A_j^2\\
q_{j,2} & = q_j\\
q_{j,3} & = \tau P_{>100\tau} A_{j}\\
q_{j,4} & = \na A_j.
\end{align*}

By the localization estimate~\eqref{loc} and the fact that $\a{\z - \ti \z} \ll 1$, we have
\[\n{L_{\ti A_i,q,\ti\z_i} + \lp_{\ti\z_i}}_{\dot X^{1/2}_{\ti\z_i} \to\dot X^{-1/2}_{\ti\z_i}} \ll Z_1,\]
If $Z_1$ is sufficiently small, then by the contraction mapping principle there are $u_i = e^{x\cd \ti \z_i} (e^{-i\phi_i} + \psi_i)$ solving $L_{\ti A_i,q_i} u_i = 0$ in $\Om$ such that
\[\n{\psi_i}_{\dot X^{1/2}_{\ti \z_i}} \ll \n{\ti \chi(F_i+G_i)}_{X^{-1/2}_{\z_i}}.\]
By Lemma~\ref{lpchia} and the Strichartz estimate~\eqref{dualstrichartz}, we have
\begin{align*}
\n{F_i}_{X^{-1/2}_{\z_i}} & \ll Z_2^4 + Z_2 + \n{A_i}_3 \n{\na a_i}_{L^2(B)} \\
&\ll Z_2^4 + Z_2 + M e^{Z_0} Z_0.
\end{align*}
On the other hand we can simply estimate $\ti \chi G$ by
\begin{align*}
\n{\ti \chi G_i}_{X^{-1/2}_{\z_i}} & \ll Z_2 +  e^{Z_0}(Z_0 + M).\
\end{align*}
Thus we have
\[\n{\psi_i}_{\dot X^{1/2}_{\ti \z_i}} \ll g(M, Z_0,Z_2),\]
where $g$ is continuous.

Now we apply the integral identity~\eqref{weakmagnetic} to the solutions $u_1$ and $u_2$ to obtain
\begin{align*}
0&=\int [i(A_1-A_2)\cd (u_1 \na  u_2 - u_2 \na  u_1) + (A_1^2 -A_2^2 + q_1-q_2)u_1  u_2]\,dx \\
&= I + II + III + IV ,
\end{align*}
where
\begin{align*}
I & =i(\z_2-\z_1) \int \chi (A_1-A_2)_{\leq 100\tau} e^{-i(\phi_1+\phi_2)} e^{i k\cd x} \,dx \\
II &= \int [(\z - \ti \z)\chi A_{\leq 100\tau} + \chi \ti \z A_{> 100 \tau} + A \na \phi + \chi q] e^{-i \phi} e^{ik\cd x} \,dx\\
III & = \int [\ti \z A a \psi + A \na a \psi + A a \na \psi + \chi q a \psi]e^{ik \cd x}\,dx\\
IV & = \int q \psi_1 \psi_2 e^{ik \cd x}\,dx.
\end{align*}
Here $k = rUe_3$. The error terms $II$-$IV$ are in schematic form. For example, the notation $\int \chi qa \psi e^{ik\cd x}$ represents a linear combination of terms $\int\chi q_{l,m} a_i \psi_j e^{ik \cd x}$.

The first expression $I$ contains the main term. We remove the exponential factor $e^{-i(\phi_1+\phi_2)}$ using Lemma~\ref{undophase}. We use the Littlewood-Paley commutator estimate $\n{[\chi,P_{\leq 100\tau}]}_{L^1\to L^1} \ll \tau^{-1}$ to control some of the errors.
\begin{align*}
I & =i( \z_2-\z_1)\cd (\chi (A_1 - A_2)_{\leq 100\tau})^\we (k)\\
&=i( \z_2-\z_1)\cd(P_{\leq 100 \tau}(A_1-A_2) + [\chi, P_{\leq 100 \tau}] (A_1-A_2))^\we (k) \\
&= i( \z_2-\z_1)\cd (A_1 - A_2)^\we (k) + O(\n{A_1-A_2}_1).
\end{align*}
Next we estimate the terms in $II$. 
\begin{align*}
\a{II} & \ll \n{aA}_{L^1} + \tau \n{A_{> 100\tau}}_{H^{-1}} \n{\chi a}_{H^1} + \n{A}_2 \n{\chi a \na \phi}_2 + \n{q}_{H^{-1}} \n{\chi a}_{H^1}\\
&\ll e^{Z_0}Z_0 M .
\end{align*}
For the terms in $III$ we use~\eqref{loc}, the duality between $\dot X^{1/2}_\z$ and $\dot X^{-1/2}_\z$, the Strichartz estimate~\eqref{dualstrichartz}, and the fact that multiplication by $e^{ik\cd x}$ is bounded in $X^{1/2}_\z$:
\begin{align*}
\a{III} & \ll (\tau \n{Aa}_{X^{-1/2}_\z} +\n{A \na a}_{X^{-1/2}_\z}+ \n{\na A \chi a}_{X^{-1/2}_\z}  + \n{q \chi a}_{X^{-1/2}_\z}) \n{\psi}_{X^{1/2}_\z} \\
&\ll (\tau^{1/2} \n{A}_2 e^{Z_0} + \n{A}_3 \n{\na a}_{L^2(B)} + Z_2) g(M,Z_0,Z_2) \\
&\ll (\tau^{1/2} Me^{Z_0} + M e^{Z_0} Z_0 + Z_2)g(M,Z_0,Z_2)
\end{align*}
For $IV$, we estimate by $\n{q}_{X^{1/2}_\z\to X^{-1/2}_\z}\n{\psi_1}_{X^{1/2}_\z} \n{\psi_2}_{X^{1/2}_\z}$. Thus
\[\a{IV} \ll Z_1 g(M, Z_0,Z_2)^2.\]
Combining all these estimates, we obtain
\begin{equation}
\label{tozero}
\a{i( \z_2-\z_1)\cd (A_1 - A_2)^\we (k) } \leq \tau^{1/2} f(M,Z_0,Z_1,Z_2),
\end{equation}
where $f$ is continuous. To conclude, we must select $\z$ such that all of the constants $Z_i$ are bounded uniformly in $\z$. 

First, we note that by Lemma~\ref{avgestimates}, Lemma~\ref{h1}, Lemma~\ref{phili} and Lemma~\ref{lil2}, we have
\[\E[Z_0\mid SO(3)] \leq C \e.\]
Define the set $V$ by
\[V = \{U \in SO(3): Z_0 < 2 C \e\},\]
By Chebyshev's inequality, we have $\Pr[V] \geq \ff 12$. Similarly, Lemma~\ref{solvability} implies that
\[\limsup_{K \to \infty} \ti \E_K[Z_1\mid V] \leq C \e.\]
Lemma~\ref{aq} implies that
\[\ti\E_K[Z_2 \mid [2^K,2^{K^2}]\times V] \leq h(Z_0,M), \]
where $h$ is continuous. Thus for sufficiently large $K$, it follows that the inequality
\[(2C \e)^{-1} Z_1 + (2h(Z_0,M))^{-1} Z_2 \leq 2\]
holds on a set $\ti V_K \subset [2^K,2^{K^2}]\times V$ with $\ti \Pr_K[\ti V_K] \geq \ff 1 4$. By choosing $\e$ small, we can ensure that $Z_1$ is sufficiently small on $\ti V_K$ that we can use the contraction mapping principle to construct $\psi_i$ as above. Furthermore, for $(\tau,U) \in \ti V_K$, the quantities $Z_i$ are all bounded independently of $\tau, U, K$. Let
\[J(r, U) = U(e_1 + i e_2) \cd (A_1 - A_2)^\we(r U e_3).\]
By~\eqref{tozero}, we have $\a{J(r,U)} \ll \tau^{-1/2}$ for all $(\tau,U)$ in the set $\ti V_K$. Integrating this inequality over all $(\tau, U)$ in $\ti V_K$ and $r$ in $[1,2]$, we have
\[\int_{SO(3)}\int_1^2 \int_{[2^K,2^{K^2}]} 1_{\ti V_K} \a{J(r, U)} (\log K)^{-1} (\tau \log \tau)^{-1} d\tau\,dr\,d\si(U) = O(e^{-K/2})\]
Let 
\[\eta_K = \int_{[2^K,2^{K^2}]} 1_{\ti V_K}\,(\log K)^{-1} (\tau \log \tau)^{-1} d\tau,\]
and note that 
\[\int_{SO(3)}\int_1^2 \eta_{\ti V_K}\,dr \,d\si(U) \sim \ti \Pr_K[\ti V_K] \geq \ff 1 4.\]
By the Banach-Alaoglu theorem, there is a sequence $K_i\to \infty$ and a function $\eta \in L^\infty(SO(3) \times [1,2])$ such that $\eta_{K_i} \wc \eta$. Since $\int \eta = \lim \int \eta_{K_i} \geq \f 1 4$, it is clear that $\eta \neq 0$. On the other hand,
\[\int \eta(r,U) \a{J(r,U)}\,dr \,d\si(U) = \lim_{i \to \infty} \int \eta_{K_i}(r,U)\a{J(r,U)}\,dr \,d\si(U)=0.\]
It follows that $J(r,U)$ vanishes on a set of positive measure. But $A_1 - A_2$ is a compactly supported function, which implies that $J(r,U)$ is analytic in $r$ and $U$. Thus we can conclude that $J(r,U) = 0$ in $\R_+\times SO(3)$. By replacing $SO(3)$ by its complement throughout the argument, we find that $J(r,U) = 0$ in $\R_+ \times (O(3)\sm SO(3))$ as well. Let $H = A_1 - A_2$. Since $J(r,U)$ vanishes uniformly, we must have $v \cd \hat H(k)=0$ whenever $v \cd k = 0$. In particular, $0 = (w \times k) \cd \hat H(k) = (\curl H)^\we(k) \cd w$ for any $w,k\in \R^n$, so $\curl H = 0$.

By Lemma~\ref{poincare}, there is a gauge transformation $\psi$ such that $A_2 = A_1 + \na \psi$, which implies that $\Ld_{A_1,q_1} =  \Ld_{A_2,q_2} = \Ld_{A_1,q_2}$. We can repeat the whole argument to obtain
\begin{align*}
0 & = \int ((q_1-q_2)  e^{i k\cd x} + \chi q a \psi e^{ik\cd x} + \chi q \psi_1 \psi_2 e^{i k\cd x})\,dx.
\end{align*}
Since $\ti E_K[\n{\chi a q}_{X^{-1/2}_\z} + \n{q}_{X^{1/2}_\z \to X^{-1/2}_\z}\mid \ti V_K] \to 0$ as $K \to \infty$, we can repeat the arguments above to show that that $(q_1 -q_2)^\we(k) = 0$ for all $k$. It follows that $q_1=q_2$.
\end{proof}

\section*{Acknowledgments}
The author would like to thank his PhD advisor Daniel Tataru for suggesting the problem, for his patient guidance and encouragement, and for many helpful suggestions. He would also like to thank Mikko Salo, Gunther Uhlmann, Russell Brown, Michael Christ and Herbert Koch for taking the time to read various versions of the manuscript and for many interesting discussions about the problem.

\bibliography{/home/boaz/Documents/library.bib}{}
\bibliographystyle{alphaabbr}
\end{document}